\documentclass{article}
\usepackage{fullpage}
\usepackage[colorlinks,backref]{hyperref}
\usepackage{graphicx}
\usepackage{amsmath,amsthm,amssymb,enumerate}
\usepackage[normalem]{ulem} 
\usepackage{euscript,mathrsfs}
\usepackage[left=3cm,right=3cm,top=3.5cm,bottom=3.5cm]{geometry}
\usepackage{color}
\usepackage{bm}
\usepackage{bbm}
\catcode`\@=11 \@addtoreset{equation}{section}

\catcode`\@=12
\usepackage{color}
\allowdisplaybreaks

\newtheorem{Theorem}{Theorem}[section]
\newtheorem{Proposition}[Theorem]{Proposition}
\newtheorem{Lemma}[Theorem]{Lemma}
\newtheorem{Corollary}[Theorem]{Corollary}

\theoremstyle{definition}
\newtheorem{Definition}[Theorem]{Definition}

\newtheorem{Remark}[Theorem]{Remark}

\newcommand{\bTheorem}[1]{
\begin{Theorem} \label{T#1} }
\newcommand{\eT}{\end{Theorem}}

\newcommand{\bProposition}[1]{
\begin{Proposition} \label{P#1}}
\newcommand{\eP}{\end{Proposition}}

\newcommand{\bLemma}[1]{
\begin{Lemma} \label{L#1} }
\newcommand{\eL}{\end{Lemma}}

\newcommand{\bCorollary}[1]{
\begin{Corollary} \label{C#1} }
\newcommand{\eC}{\end{Corollary}}

\newcommand{\bRemark}[1]{
\begin{Remark} \label{R#1} }
\newcommand{\eR}{\end{Remark}}

\newcommand{\bDefinition}[1]{
\begin{Definition} \label{D#1} }
\newcommand{\eD}{\end{Definition}}

\newcommand{\dif}{\mathrm{d}}

\newcommand{\mf}{\mathbb{F}}
\newcommand{\mr}{\mathbb{R}}
\newcommand{\prst}{\mathbb{P}}
\newcommand{\p}{\mathbb{P}}
\newcommand{\T}{\mathbb{T}}
\newcommand{\stred}{\mathbb{E}}

\newcommand{\mn}{\mathbb{N}}

\newcommand{\mt}{\mathbb{T}^3}

\DeclareMathOperator*{\esslim}{ess\,lim}

\newcommand{\bu}{\mathbf u}

\newcommand{\tor}{\mathbb{T}^3}

\newcommand{\StoB}{\left(\Omega, \mathbb{F},(\mathbb{F}_t )_{t \geq 0},  \mathbb{P}\right)}

\newcommand{\bfu}{\mathbf{u}}

\newcommand{\bfr}{\mathbf{s}}

\newcommand{\bFormula}[1]{
\begin{equation} \label{#1}}
\newcommand{\eF}{\end{equation}}

\newcommand{\dv}{\rm div}

\newcommand{\vu}{\vc{u}}
\newcommand{\vc}[1]{{\bf #1}}

\newcommand{\Div}{{\rm div}}
\newcommand{\Grad}{\nabla}

\newcommand{\tn}[1]{\mathbb{#1}}
\newcommand{\dx}{\,{\rm d} {x}}

\newcommand{\dt}{\,{\rm d} t }

\newcommand{\D}{{\rm d}}
\newcommand{\ep}{\varepsilon}

\newcommand{\R}{\mathbb{R}}

\newcommand{\E}{\mathbb{E}}

\definecolor{Cgrey}{rgb}{0.85,0.85,0.85}
\definecolor{Cblue}{rgb}{0.50,0.85,0.85}
\definecolor{Cred}{rgb}{1,0,0}
\definecolor{fancy}{rgb}{0.10,0.85,0.10}

\newcommand{\Dif}{{\rm d}}

\newcommand{\N}{\mathbb N}
\newcommand{\intTor}[1]{\int_{\tor} #1 \,\D x}

\newcommand\Cbox[2]{%
    \newbox\contentbox%
    \newbox\bkgdbox%
    \setbox\contentbox\hbox to \hsize{%
        \vtop{
            \kern\columnsep
            \hbox to \hsize{%
                \kern\columnsep%
                \advance\hsize by -2\columnsep%
                \setlength{\textwidth}{\hsize}%
                \vbox{
                    \parskip=\baselineskip
                    \parindent=0bp
                    #2
                }%
                \kern\columnsep%
            }%
            \kern\columnsep%
        }%
    }%
    \setbox\bkgdbox\vbox{
        \color{#1}
        \hrule width  \wd\contentbox %
               height \ht\contentbox %
               depth  \dp\contentbox
        \color{black}
    }%
    \wd\bkgdbox=0bp%
    \vbox{\hbox to \hsize{\box\bkgdbox\box\contentbox}}%
    \vskip\baselineskip%
}

\DeclareMathOperator*{\esssup}{ess\,sup}

\date{}

\begin{document}


\title{On weak-strong uniqueness for stochastic equations of incompressible fluid flow}

\author{
Abhishek Chaudhary \footnotemark[2]
\and Ujjwal Koley \footnotemark[3]
} 

\date{\today}

\maketitle


\medskip
\centerline{$^\dagger$ Centre for Applicable Mathematics, Tata Institute of Fundamental Research}
\centerline{P.O. Box 6503, GKVK Post Office, Bangalore 560065, India}\centerline{abhi@tifrbng.res.in}

\medskip
\centerline{$^\ddagger$ Centre for Applicable Mathematics, Tata Institute of Fundamental Research}
\centerline{P.O. Box 6503, GKVK Post Office, Bangalore 560065, India}
\centerline{ujjwal@math.tifrbng.res.in}

\begin{abstract}
We introduce a novel concept of \emph{dissipative measure-valued martingale solution} to the stochastic Euler equations describing the motion of an \emph{inviscid incompressible} fluid. These solutions are characterized by a parametrized Young measure and a concentration defect measure in the total energy balance. Moreover, they are weak in the probablistic sense i.e., the underlying probablity space and the driving Wiener process are intrinsic part of the solution. In a significant departure from the existing literature, we first exhibit the relative energy inequality for the incompressible Euler equations driven by a \emph{multiplicative} noise, and then demonstrate pathwise weak-strong uniqueness principle. Finally, we also provide a sufficient condition, \'a la Prodi \cite{prodi} $\&$ Serrin \cite{serrin}, for the uniqueness of weak martingale solutions to stochastic Naiver--Stokes system in the class of finite energy weak martingale solutions.

\end{abstract}


\tableofcontents

\section{Introduction}
\label{P}
In this paper, we introduce a notion of \emph{dissipative measure-valued  solution} for the stochastically forced system of the \emph{incompressible
Euler system} describing by the velocity vector field $\vu$ of a fluid and the scalar pressure field $p$. The system of equations reads
\begin{equation}
\begin{cases} 
\D \vu(t,x) + \left[ \Div (\vu(t,x) \otimes \vu(t,x))+  \Grad p(t,x)\right]  \dt  =  \mathcal{G} (\vu(t,x)) \,\D W(t), &\quad \mbox{in}\,\, \Pi_T, \label{P1} \\ 
\Div \,\vu(t,x) = 0,&\quad \mbox{in}\,\, \Pi_T, \\
\vu(0, x) = \vu_0(x), &\quad \mbox{in}\,\, \mathbb{T}^3, 
\end{cases}
\end{equation}
where $\Pi_T:=\mathbb{T}^3 \times (0,T)$ with $T>0$ fixed, $u_0$ is the given random initial
function with sufficient spatial regularity to be specified later. Let $\big(\Omega, \mathbb{F}, \mathbb{P}, (\mathbb{F}_t )_{t\ge0} \big)$ be a stochastic basis, where $\big(\Omega, \mathbb{F}, \mathbb{P} \big)$ is a probability space and $(\mathbb{F}_t)_{t \ge 0}$ is a complete filtration with the usual assumptions. We assume that $W$ is a cylindrical Wiener process defined on the probability space $(\Omega,\mathbb{F},\p)$, and the coefficient $\mathcal{G}$ is generally nonlinear and satisfies suitable growth assumptions (see Section \ref{E} for the complete list of assumptions). In particular, the map $ \vu \mapsto \mathcal{G}(\vu)$ is  a Hilbert space valued function signifying the \emph{multiplicative} nature of the noise. 

The Euler equations \eqref{P1} are the classical model for the motion of an inviscid, incompressible fluid. The presence of stochastic terms in the governing equations accounts for numerical, physical and empirical uncertainties in various real life applications. The theory for deterministic counterpart of \eqref{P1} has experienced a substancial progress in past decade, and has reached some level of maturity thanks to pioneering work by De Lellis $\&$ Szekelyhidi \cite{DelSze3,DelSze}. In a nutshell, these results show that for suitable (large class of) initial data there exists \emph{infinitely many} weak solutions, even if the solution satisfies an \emph{entropy condition}. Indeed, it was Scheffer \cite{Sch} (see also \cite{Shn}) who first constructed a nontrival weak solution of the two dimensional incompressible Euler equation with compact support in time. In other words, recent results by De Lellis $\&$ Szekelyhidi, and others \cite{buck01,buck02,chio01} suggest that non-uniqueness of weak solutions to incompressible Euler equations in several space dimensions is a fact of life. However, for \emph{general} initial data, the existence of global-in-time weak solutions is still unknown. In quest for a global-in-time solution, we recall the framework of \emph{measure-valued} solutions, as introduced by Diperna $\&$ Majda \cite{Majda} (see also \cite{DL}) for the incompressible Euler equations. Despite being a weaker notion of solution, thanks to the work by Brenier et. al. \cite{brenier}, it is well-known that these measure-valued solutions of incompressible Euler equations enjoy a remarkable \emph{weak-strong uniqueness} property. Moreover, recent work by Szekelyhidi $\&$ Wiedemann \cite{emil_01} confirms that the notions of weak solutions and measure-valued solutions coincide for incompressible Euler equations.

 In the stochastic set-up, existence of \emph{pathwise} strong solutions (defined upto a stopping time) for incompressible Euler equations \eqref{P1} driven by a \emph{multiplicative noise}, in a three dimensional smooth bounded domain, was established by Glatt-Holtz $\&$ Vicol \cite{GHVic}. Moreover, a stochastic variant of deterministic results, as developed by Buckmaster $\&$ Vicol in \cite{buck01,buck02}, for the stochastic incompressible Euler system have been recently established by Hofmanova. et. al. in \cite{hof}. In fact, by making use of method of convex integration, they have constructed  infinitely many solutions to the incompressible Euler system with a random forcing. Note that, although measure-valued solution for the deterministic counterpart of \eqref{P1} has a long and intense history (see Lions \cite{lions}), their formulation (in the multiplicative case) seems rather intricating in the stochastic setting. Indeed, there has been a number of attemps to define a suitable notion of measure-valued solutions for the stochastic incompressible Euler equations driven by \emph{additive} noise, starting from the work of Kim \cite{Kim2}, Breit $\&$ Moyo \cite{BrMo}, and most recently by Hofmanova et. al. \cite{hof}, where the authors introduced a class of dissipative solutions which allowed them to demostrate weak-strong uniqueness property and non-uniqueness of solutions in law. However, none of the above mentioned framework can be applied to \eqref{P1} since the driving noise is \emph{multiplicative} in nature.

In this paper, our first objective is to introduce a novel framework of \emph{dissipative measure-valued} solutions for incompressible Euler equations \eqref{P1} driven by a \emph{multiplicative} stochastic perturbation of Nemytskii type, which will facilitate to demonstrate weak-strong uniquness principle. These solutions are measure-valued solutions of \eqref{P1} augmented with an appropriate form of energy inequality. The main advantage of this class of solutions is that, for any finite energy initial data, they can be shown to exist globally-in-time. 
The existence of such solutions is addressed through the well known vanishing viscosity method with the help of a stochastic compactness argument combining Prohorov theorem, and Skorokhod-Jakubowski \cite{Jakubowski} representation theorem. In other words, measure-valued solutions for \eqref{P1} is obtained through a sequence of solutions of Navier--Stokes system subject to stochastic forcing given by
\begin{align} \label{P1NS}
\D \vu_{\varepsilon} + \left[ \Div ( \vu_{\varepsilon} \otimes \vu_{\varepsilon}) +  \Grad p \right] \,dt  &= {\varepsilon}\,\Delta \vu \,dt + \mathcal{G} ( \vu_{\varepsilon}) \,\D W,
\\
 \Div \, \vu_{\varepsilon} \,&= 0 \notag.
\end{align}
There is a vast literature on the mathematical theory for stochastic perturbations of Navier--Stokes equations \eqref{P1NS}, being first initiated by Flandoli $\&$ Gatarek in \cite{GHVic} (see also \cite{Debussche}), where the global-in-time existence of \emph{weak martingale solutions} is shown. From the PDE standpoint, 
these solutions are weak, i.e., derivatives only exists in the sense of distributions, and from probabilistic point of view, these solutions are also weak in the sense that the driving noise and associated filtration are part of the solution.

 The primary difficulty in comparison to existing works, for example \cite{BrMo, Kim2}, lies in the successful identification (in the limit) of martingale terms involving nonlinear noise coefficient, which in turn plays a pivotal role in the proof of weak (measure-valued)-strong uniqueness principle. Indeed, since the martingale solutions of approximate equations \eqref{P1NS} are not unique, associated filtration depends on the approximation parameter $\varepsilon$, and passing to the limit in martingale terms seems delicate. We can only establsih that the limit object is a martinagle, without knowing its explicit structure. However, a key observation reveals that only a knowledge about the cross variation of a martingale solution with a strong solution is sufficient to exhibit weak-strong uniqueness property. This observation is encoded in the Definition~\ref{def:dissMartin}, by stipulating the correct cross variation between a martingale measure-valued solution and a given smooth process. The cross variation term resembles ``noise-noise'' interaction term appears in ``doubling of variables'' argument, \`{a} la Kru\v{z}kov, see \cite{BhKoleyVa, BhKoleyVa_01, BisKoleyMaj, Koley1, Koley2}. Recently, a concept of dissipative measure-valued solutions to stochastic compressible Euler equations has been introduced by Hofmanova et. al. in \cite{MKS01}, and used to prove convergence of a numerical scheme for stochastic compressible Euler equations by Chaudhury et. al. in \cite{K1}. However, the present work differs significantly from \cite{MKS01,K1}, due to the inherent challenges posed by the divergence-free condition associated to \eqref{P1}.

Our second objective is to display a sufficient condition for the uniqueness of weak martingale solutions to stochastic Naiver--Stokes equations \eqref{P1NS} in class of finite energy weak martingale solutions.  
Note that the uniqueness of (Leray-Hopf) weak solutions for deterministic counterpart of \eqref{P1NS} is an outstanding open problem. Therefore, there has been a lot of growing interest in searching for sufficient conditions for uniqueness of weak solutions, starting with the celebrated works by Prodi \cite{prodi} and Serrin \cite{serrin}. Due to the limitations of general uniqueness result in the stochastic set-up also, we follow Prodi $\&$ Serrin (see also Wiedemann \cite{emil}), to provide a sufficient condition for uniqueness by exploiting weak-strong uniqueness property of solutions to stochastic incompressible Navier--Stokes equations \eqref{P1NS}.

A brief description of the organization of the rest of this paper is as follows: Section~\ref{E} outlines all relevant underlying mathematical/technical framework, and list of conditions imposed on noise coefficient. We then describe solution concepts and display main results. In Section~\ref{proof1}, we establish a priori estimates and demonstrate convergence of approximate solutions, using stochastic compactness, to show existence of dissipative measure-valued martingale solutions to \eqref{P1}. In light of a suitable relative energy inequality for incompressible Euler system \eqref{P1}, we establish weak (measure-valued) -- strong uniqueness principle for \eqref{P1} in Section~\ref{proof2}, while in Section~\ref{NSE}, we provide a sufficient condition for uniqueness of weak martingale solutions to stochastic incompressible Naiver--Stokes equations \eqref{P1NS}, under additional conditions on weak martingale solutions.

\section{Technical Framework and Main Results} 
\label{E}

In this section, we recapitulate some of the relevant mathematical tools to be used in the subsequent analysis and state main results of this paper. To begin with, we fix an arbitrary large time horizon $T>0$. Throughout this paper we use the letters C, K etc. to denote various generic constants independent of approximation parameters, which may change line to line along the proofs. Explicit tracking of the constants could be possible but it is cumbersome and avoided for the sake of the reader.

\subsection{Analytic framework}
Let us denote the Sobolev space $H^s(\T^3)$, for $s \in \R$, as the set of tempered distributions for which the norm
$$
\| u\|^2_{H^s(\T^3)}:= \int_{\T^3} \big(  1 + |\xi|^2\big)^s |\hat{u}(\xi)|^2\, \D \xi, 
$$
is finite. Here $\hat{u}$ denotes the Fourier transform of $u$. Let $C_{\text{div}}^\infty(\mathbb{T}^3)$ and $L_{\text{div}}^2(\mathbb{T}^3)$ be the spaces of infinitely differentiable $3$-dimensional vector fields $u$ on $\mathbb{T}^3$ satisfying $\nabla \cdot \,u=0$, and closure of $C_{\text{div}}^\infty(\mathbb{T}^3)$ with respect to $L^2$-norm respectively. In other words,
\begin{align*}
C_{\text{div}}^\infty(\mathbb{T}^3) &:=\Big\{\bm{\varphi}\in C^{\infty}(\mathbb{T}^3):\nabla\cdot\bm{\varphi}=0\Big\}, \\
L_{\text{div}}^2(\mathbb{T}^3) &:=\Big\{\bm{\varphi}\in L^{2}(\mathbb{T}^3):\nabla\cdot\bm{\varphi}=0\Big\}.
\end{align*}
In a similar fashion, we denote by $ H_{\text{div}}^{\alpha}(\mathbb{T}^3)$ the closure of $C_{\text{div}}^\infty(\mathbb{T}^3)$ in $H^{ \alpha}(\mathbb{T}^3;\R^3)$, for $\alpha\,\ge\,0$.
Identityfying $L^2_{\dv}(\mathbb{T}^3)$ with its dual space $(L^2_{\dv}({\mathbb{T}^3}))'$ and identfying $(L^2_{\dv}({\mathbb{T}^3}))'$ with a subspace of $H^{-\alpha}(\T^3)$(the dual space of $H^\alpha(\T^3)$), we have $H_{\dv}^\alpha(\T^3)\subset L^2_{\dv}(\mathbb{T}^3)\subset H_{\dv}^{-\alpha}(\T^3)$, and we can denote the dual pairing between $H^\alpha_{\dv}$ and $H^{-\alpha}_{\dv}$ by $\langle \cdot, \cdot\rangle$ when no confusion may arise, see \cite{FlandoliGatarek}.

Moreover, we set  $D(A):=H_{\dv}^2(\mathbb{T}^3)$, and define the linear operator $A:D(A)\subset L^2_{\dv}(\mathbb{T}^3)\to L^2_{\dv}(\mathbb{T}^3)$ by $Au=-\Delta u$.
We then define the bilinear operator $B(u,v):H^1_{\dv}\times H^{1}_{\dv} \to H^{-1}_{\dv}$ as
$$
\langle B(u,v), z\rangle:=\int_{\mathbb{T}^3}z(x)\cdot(u(x)\cdot\nabla)v(x)dx, \,\, \mbox{for all}\,\,z\in H^1_{\dv}(\mathbb{T}^3).
$$
Note that the bilinear operator $B$ can be extended to a continuous operator
$$
B:L^2_{\dv}(\mathbb{T}^3)\times L^2_{\dv}(\mathbb{T}^3)\to D(A^{-\alpha})=H^{-2\alpha}_{\dv}
$$
for certain $\alpha\,\textgreater\,1 $, for details consult \cite{FlandoliGatarek}. A straightforward computation using incompressiblity condition reveals that
\begin{align}
\label{property}
\langle B(u,v),z\rangle=-\langle B(u,z),v\rangle=-\langle u\otimes v, \nabla z \rangle
\end{align}
for all $u,v\in H^1_{\dv}(\mathbb{T}^3)$ and $z\in C_{\text{div}}^\infty(\mathbb{T}^3)$.

An important consequence of elliptic theory is the existence of the Helmholtz decomposition. It allows to decompose any vector-valued function in $L^2(\mathbb{T}^3;\R^3)$ into a divergence free part and a gradient part. In other words, we have the following decomposition
$$
L^2(\mathbb{T}^3)=L^2_{\dv}(\mathbb{T}^3)\oplus (L^2_{\dv}(\mathbb{T}^3))^\perp,
$$
where we denote
$$
(L^2_{\dv}(\mathbb{T}^3))^\perp:=\Big\{\mathbf{u}\in L^2(\T^3;\R^3) \,|\, \mathbf{u}=\nabla \psi,\, \psi\in H^{1}(\T^3;\R)\Big\}
$$ 
The Helmholtz decomposition is defined by
$$
\mathbf{u}=\mathcal{P}_H \mathbf{u}+\mathcal{Q}_H\mathbf{u}, \,\,\text{for any }\,\mathbf{u}\in L^2(\mathbb{T}^3),
$$
where $\mathcal{P}_H$ denotes the projection operator from $L^2(\mathbb{T}^3)$ to $L^2_{\dv}(\mathbb{T}^3)$, and 
 $\mathcal{Q}_H:=Id-\mathcal{P}_H$ denotes the projection operator from $L^2(\mathbb{T}^3)$ to $(L^2_{\dv}(\mathbb{T}^3))^\perp$. Note that this decomposition is orthogonal with respect to $L^2(\mathbb{T}^3)$-inner product.
By property of projection operator $\mathcal{P}_H$, we have for $\mathbf{u}\in L^2(\mathbb{T}^3)$
\begin{align}
\label{property1}
\langle \mathcal{P}_H\mathbf{u},\psi\rangle=\langle \mathbf{u}, \psi \rangle,\,\, \text{for all}\,\, \psi\in C_{\text{div}}^\infty(\mathbb{T}^3),
\end{align}
Finally, we recall a compact embedding result from Flandoli $\&$ Gatarek \cite[Theorem 2.2]{FlandoliGatarek}. To state the result, let us first denote $K$ to be a separable Hilbert space. Given $q\,\textgreater\,1$, $\gamma \in (0,1)$, let $W^{\gamma,q}(0,T;K)$ denotes a $K$-valued Sobolev space which is characterized by its norm
$$
\| u\|^q_{W^{\gamma,q}(0,T;K)}:= \int_0^T \| u(t)\|^q_{K}\,dt + \int_0^T \int_0^T \frac{\| u(t)-u(s)\|^q_K}{|t-s|^{1+q\gamma}}\,dt\,ds.
$$
The following compact embedding result follows from Flandoli $\&$ Gatarek \cite[Theorem 2.2]{FlandoliGatarek}.
\begin{Lemma}
\label{comp}
If $H_1 \subset H_2$ are two Banach spaces with compact embedding, and real numbers $\gamma \in(0,1)$, $q\,\textgreater\,1$ satisfy $\gamma q>1$, then the embedding
$$
W^{\gamma,q}(0,T;H_1) \subset C([0,T]; H_2)
$$
is compact.
\end{Lemma}

\subsubsection{Young measures, concentration defect measures}
\label{ym}

In this subsection, we recall the notion of Young measures and related results used in this manuscript. For a detail overview on this topic, we refer to Balder \cite{Balder}. To begin with, let us denote by $\mathcal{M}_b(F)$, the space of bounded Borel measures on a generic set $F$ equipped with the norm given by the total variation of measures. It is well-known that it is the dual space to the space of continuous functions vanishing at infinity $C_0(F)$ with respect to the supremum norm. Moreover, let us denote by $\mathcal{P}(F)$, the space of probability measures on $F$. 

To define Young measure, let us first fix a sigma finite measure space $(Y, \mathcal{N}, \mu)$. A Young measure from $Y$ into $\R^P$ is a weakly measurable function $\mathcal{V}: Y \rightarrow \mathcal{P}(\R^P)$. In other words, the map $y \rightarrow \mathcal{V}_{y}(S)$ is $\mathcal{N}$-measurable for every Borel set $S$ in $\R^P$. For our purpose, let us recall the following probabilistic generalization of the well-known classical result on Young measures. For a proof, we refer to the monograph by Breit et. al. \cite[Section~2.8]{BrFeHobook}.
\begin{Lemma}\label{lem001}
Let $P, Q \in\mathbb{N}$, $\mathcal{D} \subset \R^Q\times (0,T) $ and let $({\bf V}_m)_{m \in \N}$, ${\bf V}_m: \Omega\times\mathcal{D}  \to \R^P$,  be a sequence of random variables such that
$$
\E \Big[ \|{\bf V}_m \|^p_{L^p(\mathcal{D})}\Big] \le C, \,\, \text{for a certain}\,\, p\in(1,\infty).
$$
Then, on the complete probability space $\big([0,1], \overline{\mathcal{B}[0,1]}, \mathcal{L}_{\R} \big)$, there exists a new subsequence $(\tilde {\bf V}_m)_{m \in \N}$ (not relabeled) and a
family ${\lbrace \mathcal{\tilde V}^{\omega}_{x} \rbrace}_{x \in \mathcal{D}}$ of  parametrized
random probability measures on $\R^P$. These random probability measures can be treated as random variables taking values in $\big(L_{w^*}^{\infty}(\mathcal{D}; \mathcal{P}(\R^P)), w^* \big)$. Moreover, the random variables ${\bf V}_m$ has the same law as $\tilde {\bf V}_m$, i.e.
$
{\bf V}_m \sim_{d} \tilde {\bf V}_m,
$
and  the following property holds: given any Carath\'eodory function $H=H(x, Z), x \in \mathcal{D}, Z \in \R^P$, such that
$$
|H(x,Z)| \le C(1 + |Z|^q), \quad 1 \le q < p, \,\,\text{uniformly in}\,\,x,
$$
implies $\mathcal{L}_{\R}$ almost surely,
$$
H(\cdot, \tilde {\bf V}_m) \rightharpoonup \overline{H}\,\,\text{in}\,\, L^{p/q}(\mathcal{D}), 
\,\, \text{where}\,\, \overline{H}(x) = \langle \mathcal{\tilde V}^{\omega}_{(\cdot)}; H(x,\cdot)\rangle := \int_{\R^P} H(x, z)\,\D \mathcal{\tilde V}^{\omega}_{x}(z), \,\,\text{for a.a.}\,\, x \in \mathcal{D}.
$$
\end{Lemma}
\noindent Usually, Young measure theory plays an important role while extracting limits of bounded continuous functions. However, in our context, we have to deal with nonlinear function $H$ which may not be continuous, but instead, enjoys the following bound: 
\begin{equation*}
\E \Big[ \|H({\bf V}_m)\|^p_{L^1(\mathcal{D})}\Big] \le C, \,\, \text{for a certain}\,\, p\in(1,\infty), \, \mbox{uniformly in } m.
\end{equation*}
In fact, in this situation, it is customary to embed the function space $L^1(\mathcal{D})$ into the space of bounded Radon measures $\mathcal{M}_b(\mathcal{D})$ to characterize the limit object. Indeed, we can infer that $\p$-a.s.
\begin{align*}
\mbox{weak-* limit in} \, \mathcal{M}_b(\mathcal{D})\,\, \mbox{of} \, \,H({\bf V}_m) = \langle \mathcal{\tilde V}^{\omega}_{x}; H\rangle \,dx+ H_{\infty},
\end{align*}
where $H_{\infty} \in \mathcal{M}_b(\mathcal{D})$, and $H_{\infty}$ is called \emph{concentration defect measure (or concentration Young measure)}. It is worth noting that, $\p$-a.s. $\langle \mathcal{\tilde V}^{\omega}_{x}; H \rangle $ is finite for a.e. $x\in \mathcal{D}$, thanks to a classical
truncation error analysis and Fatou's lemma which yield $\| \langle \mathcal{\tilde V}^{\omega}_{(\cdot)}; H\rangle\|_{L^1(\mathcal{D})} \leq C$, $\p$ almost surely. For our purpose, we shall repeatedly use the following crucial lemma concerning the concentration defect measure. A  proof of this lemma can be furnished, modulo cosmetic changes, using the same arguments given by Feireisl et. al \cite[Lemma 2.1]{Fei01}.

\begin{Lemma}
\label{lemma001}
Let $\{{\bf V}_m\}_{m> 0}$, ${\bf V}_m: \Omega \times \mathcal{D} \rightarrow \mathbb{R}^P$ be a sequence generating a Young measure $\{\mathcal{V}^{\omega}_y\}_{y\in \mathcal{D}}$, where $\mathcal{D}$ is a measurable set in $\mathbb{R}^Q \times (0,T)$. Let $H^1: \mathbb{R}^P \rightarrow [0,\infty)$
be a continuous function such that
\begin{equation*}
\sup_{m >0} \E\Big[\|H^1({\bf V}_m)\|^p_{L^1(\mathcal{D})} \big]< +\infty, \, \text{for a certain}\,\, p\in(1,\infty),
\end{equation*}
and let $H^2$ be a continuous function such that
\begin{equation*}
H^2: \mathbb{R}^P \rightarrow \mathbb{R}, \quad |H^2(\bm{z})|\leq H^1(\bm{z}), \mbox{ for all } \bm{z}\in \mathbb{R}^P.
\end{equation*}
Let us denote $\p$-a.s.
\begin{equation*}
{H^1_{\infty}} := {\widetilde{H^1}}- \langle \mathcal{\tilde V}^{\omega}_{y}, H^1(\textbf{v}) \rangle \,dy, \quad 
{H^2_{\infty}} := {\widetilde{H^2}}- \langle \mathcal{\tilde V}^{\omega}_{y}, H^2(\textbf{v}) \rangle \,dy.
\end{equation*}
Here ${\widetilde{H^1}}, {\widetilde{H^2}} \in \mathcal{M}_b(\mathcal{D})$ are weak-$*$ limits of $\{H^1({\bf V}^m)\}_{m > 0}$, $\{H^2({\bf V}^m)\}_{m > 0}$ respectively in $\mathcal{M}_b(\mathcal{D})$. Then $|H^2_{\infty}| \leq H^1_{\infty}$ almost surely.
\end{Lemma}

\subsection{Basics of stochastic framework}

Here we briefly recall some aspects of the theory of stochastic analysis which are pertinent to the present work. We start by fixing a stochastic basis $(\Omega,\mf,(\mf_t)_{t\geq0},\prst, W)$ with a complete, right-continuous filtration. Here $(\Omega,\mf, \prst)$ is a complete probability space, and the stochastic process $W$ is a cylindrical $(\mf_t)$-Wiener process defined on an auxiliary separable Hilbert space $\mathfrak{U}$. It is formally given by the expansion
$$
W(t)=\sum_{k\geq 1} e_k W_k(t),
$$
where the elements $\{ W_k \}_{k \geq 1}$ are a sequence of mutually independent one dimensional standard Brownian motions relative to $(\mf_t)_{t\geq0}$ and $\{e_k\}_{k\geq 1}$ is a complete orthonormal basis of $\mathfrak{U}$.
To define the stochastic integral featured in \eqref{P1}, with nonlinear diffusion coefficient $\mathcal{G}$, assume that $\bfu: \Omega \times [0,T]\time\Omega\to L^2(\mt)$ is predictable, and let $\,\mathcal{G}({\bf u}):\mathfrak{U}\rightarrow L^2(\mt)$ be defined as follows
$$\mathcal{G}({\bf u})e_k=\mathbf{G}_k({\bf u}(\cdot)).$$
The coefficients $\mathbf{G}_{k}:\mr^3\rightarrow\mr^3$ are $C^1$-functions that satisfy
\begin{align}
\sum_{k\ge1}|\mathbf{G}_k({\bf u})|^2 \le D_0 (1+|{\bf u}|^2),\label{FG1}\\
\sum_{k\ge1}|\mathbf{G}_k({\bf u})-\mathbf{G}_k({\bf v})|^2\le D_1 |{\bf u}-\textbf{v}|^2,
\label{FG2}
\end{align}
for some real constants $D_0$, and $D_1$. Then the stochastic integral
\[
\int_0^t \mathcal{G}(\vu) \ {\rm d} W = \sum_{k \geq 1}\int_0^t \vc{G}_k (\vu) \ {\rm d} W_k
\]
is a well-defined $(\mf_t)$-martingale taking values in $L^2(\T^3)$.
Finally, since $W(t)=\sum_{k\geq 1} e_k W_k(t)$ does not converge in $\mathfrak{U}$, we define the auxiliary space $\mathfrak{U}_0\supset\mathfrak{U}$ via
$$\mathfrak{U}_0:=\bigg\{v=\sum_{k\geq1}\beta_k e_k;\;\sum_{k\geq1}\frac{\beta_k^2}{k^2}<\infty\bigg\},$$
according to the norm
$$\|v\|^2_{\mathfrak{U}_0}=\sum_{k\geq1}\frac{\beta_k^2}{k^2},\quad v=\sum_{k\geq1}\beta_k e_k.$$
Observe that the embedding $\mathfrak{U}\hookrightarrow\mathfrak{U}_0$ is Hilbert--Schmidt. Moreover, the trajectories of the Brownian motion $W$ are almost surely in $C([0,T];\mathfrak{U}_0)$, thanks to a standard martingale argument.

In order to state the existence of pathwise strong solution for stochastic incompressible Euler equations, we next describe the conditions imposed on the diffusion coefficient $\mathcal{G}$. For details, we refer to the paper by Glatt-Holtz $\&$ Vicol \cite{GHVic}. Although the below mentioned conditions appear to be rather involved, but they cover many realistic stochastic models. In what follows, let $L_2$ denotes the usual space of Hilbert-Schmidt operators, and for $ p\ge 2, m\ge0$, define 
\begin{align*}
\mathbb{L}^{m,p}=\bigg\{\sigma:\mathbb{T}^N\to\,L_2 \,\Big |\,\sigma_k(\cdot)=\sigma(\cdot)e_k\in W^{m,p},  \,\,\text{and} \sum_{|\beta|\,\le\,m}\int_{\mathbb{T}^3}|\partial^\beta \sigma_k|^p_{L_2}\, \D x\,\textless\,\infty\bigg\},
\end{align*}
which is a Banach space endowed with the norm 
\begin{align*}
\|\sigma\|_{\mathbb{L}^{m,p}}:=\sum_{|\beta|\le\,m}\int_{\mathbb{T}^3}|\partial^\beta \sigma|_{L_2}^p \, \D x =\sum_{|\beta|\le\,m}\int_{\mathbb{T}^3}\bigg(\sum_{k\ge1}|\partial^\beta\sigma_k|^2\bigg)^{p/2}\, \D x.
\end{align*} 
Consider any pair of Banach spaces $X,Y$. For an increasing, locally bounded function $\gamma(\cdot) \ge 1$, we denote the space of locally bounded maps
\begin{align*}
\text{Bnd}_{u,\text{loc}}(X, Y):=\bigg\{{\mathcal{G}}\in C(X;Y):\|{\mathcal{G}}(x)\|_Y\le \gamma(\| x\|_{L^{\infty}}) (1+\|x\|_X),\,\forall\,x\in X \bigg\}.
\end{align*}
In addition, we also define the space of locally Lipschitz functions,
\begin{align*}
\text{Lip}_{u,\text{loc}}(X, Y)=\bigg \{{\mathcal{G}}\in\text{Bnd}_{u,\text{loc}}(X,Y):
\|\mathcal{G}(x)-\mathcal{G}(y)\|_Y\le\,\gamma \Big(\| x\|_{L^{\infty}} + \| y\|_{L^{\infty}}\Big)\,\|x-y\|_X,\forall\, x,y \in X\bigg\}
\end{align*}
For the statement of local pathwise existence result (cf. Theorem \ref{existence of strong solution}), we shall fix $p\ge 2$ and an integer $m\,\textgreater\, {3}/{p}+1$, and suppose that 
\begin{align}\label{extra}
	\mathcal{G}\in\text{Lip}_{u,\text{loc}}(L^p, &\mathbb{L}^{0,p})\cap\text{Lip}_{u,\text{loc}}(W^{m+1,p}, \mathbb{L}^{m+1,p})\cap\text{Lip}_{u,\text{loc}}(W^{m+5,2}, \mathbb{L}^{m+5,2}).
	\end{align}
	
Relating to the convergence of approximate solutions, strong convergence in $\omega$ variable plays a pivotal role. To that context, we need Skorokhod embedding theorem, delivering a new probability space and new random variables, with the same laws as the original ones, converging almost surely. However, for technical reasons, we have to use a modified version of classical Skorokhod embedding theorem \cite[Corollary 2]{BrzezniakHausenblasRazafimandimby} which is stated below.
\begin{Theorem}
\label{newth}
Let $(\Omega, \mathcal{F},\mathbb{P})$ be a probability space and $H_1$ be a separable metric space. Let $H_2$ be a quasi-polish space (i.e., there is a sequence of continuous functions $h_n:H_2\to[-1,1]$ that separates points of $H_2$). Assume that $\mathcal{B}(H_1)\otimes\mathcal{H}_2$ is a sigma algebra associated with the product space $H_1\times H_2$, where $\mathcal{H}_2$ is the sigma algebra generated by the sequence of continuous functions $\{h_n\}_{n=1}^{\infty}$. Let $U_n:\Omega\to H_1\times H_2$, $n\in\mathbb{N}$, be a family of random variables, such that the sequence $\{\mathcal{L}aw(U_n):n\in\mathbb{N}\}$ is weakly convergent on $H_1\times H_2$.
For $k=1,2$, let $\pi_i:H_1\times H_2$ be the projection onto $H_i$, i.e.,
$$U=(U_1,U_2)\in H_1\times H_2\mapsto\pi_i(U)=U_i\in H_i.$$ 
Finally, let us assume that there exists a random variable $X:\Omega\to H_1$ such that $\mathcal{L}aw(\pi_1 (U_n))=\mathcal{L}aw(X),\,\forall\,n\in\mathbb{N}$.
Then, there exists a probability space $(\tilde{\Omega},\tilde{\mathcal{F}},\tilde{\mathbb{P}})$, a family of $H_1\times H_2$-valued random variables $\{\tilde{U}_n:n\in\mathbb{N}\}$, on $(\tilde{\Omega},\tilde{\mathcal{F}},\tilde{\mathbb{P}})$ and a random variable $\tilde{U}:\tilde{\Omega}\to H_1\times H_2$ such that
\begin{enumerate}
\item [(a)] $\mathcal{L}aw(\tilde{U}_n)=\mathcal{L}aw(U_n),\,\forall\,n\in\mathbb{N};$
\item [(b)]
$\tilde{U}_n\to \tilde{U}\,\text{in}\,H_1\times H_2,\,\mathbb{P}-\text{a.s.}$
\item [(c)]
$\pi_1(\tilde{U}_n)(\tilde{w})=\pi_1(\tilde{U})(\tilde{w}),\,\forall\,\tilde{w}\in\tilde{\Omega}. $
\end{enumerate}
\end{Theorem}

Finally, we recall the celebrated Kolmogorov continuity thereom related to the existence of continuous modifications of stochastic processes.

\begin{Lemma}
\label{lemma01}
Let $Z={\lbrace Z(t) \rbrace}_{t \in [0,T]} $ be a real-valued stochastic process defined on a complete filtered  probability space $(\Omega,\mf,(\mf_t)_{t\geq0},\prst)$. Suppose that there are constants $a >1, b >0$, and $C>0$ such that for all $s,t \in [0,T]$,
\begin{align*}
\E[|Z(t)-Z(s)|^a] \le C |t-s|^{1 + b}.
\end{align*}
Then there exists a continuous modification of the stochastic process $Z$ and the paths of $Z$ are $c$-H\"{o}lder continuous, for every $c \in [0, \frac{b}{a})$.
\end{Lemma}

\subsection{Stochastic incompressible Euler equations}

As we mentioned before, we are primarily interested in establishing weak (measure-valued)--strong uniqueness principle for dissipative measure-valued solutions to \eqref{P1}. Since such an argument requires the existence of strong solution, therefore, we first recall the notion of local strong pathwise solution for stochastic incompressible Euler equations. We remark that such a solution can be constructed on any given stochastic basis, that is, solutions are probablistically strong, and satisfies the underlying equation \eqref{P1} pointwise (not only in the sense of distributions), that is, solutions are srtong from the PDE standpoint. Existence of such a solution was first established by Glatt-Holtz $\&$ Vicol in \cite{GHVic}. 
\begin{Definition}[Local strong pathwise solution] 
\label{def:strsol}
Let $\StoB$ be a stochastic basis with a complete right-continuous filtration, and ${W}$ be an $(\mathbb{F}_t) $-cylindrical Wiener process. Let $\bfu_0$ be a $W^{m,p}(\T^3)$-valued $\mathbb{F}_0$-measurable random variable, and let $\mathcal{G}$ satisfy \eqref{extra}. Then
$(\vu,\mathfrak{t})$ is said to be a local strong pathwise solution to the system \eqref{P1} provided
\begin{enumerate}[(a)]
\item $\mathfrak{t}$ is an a.s. strictly positive  $(\mathbb{F}_t)$-stopping time;
\item the velocity $\vu$ is a $W_{\dv}^{m,p}(\mt)$-valued $(\mathbb{F}_t)$-predictable measurable process satisfying
$$ \vu(\cdot\wedge \mathfrak{t}) \in   C([0,T]; W^{m,p}_{\Div}(\mt))\quad \mathbb{P}\text{-a.s.};$$
\item for all $t\,\ge\,0$,
\begin{equation}
\begin{split}
\vu (t \wedge \mathfrak{t})&= \vu_0- \int_0^{t \wedge \mathfrak{t}}\mathcal{P}_H(\vu \cdot\nabla\vu )\dif s + \int_0^{t \wedge \mathfrak{t}}\mathcal{P}_H{\mathcal{G}}(\vu ) \ \D W.
\end{split}
\end{equation}
\end{enumerate}
\end{Definition}
\begin{Remark}
By the property of Projection operator \eqref{property1}, we can recast the item $(c)$ in Definition~\ref{def:strsol} as follows:
\begin{enumerate}
\item[$(c^\prime)$] for all $t\in[0,T]$,
\begin{equation}
\begin{split}
\langle\vu (t \wedge \mathfrak{t}),\psi\rangle&=\langle \vu_0, \psi\rangle- \int_0^{t \wedge \mathfrak{t}}\langle \vu \cdot\nabla\vu ,\psi \rangle\dif s + \int_0^{t \wedge \mathfrak{t}}\langle{\mathcal{G}}(\vu ),\psi\rangle \ \D W,
\end{split}
\end{equation}
	for all ${\psi}\in C_{\text{div}}^\infty(\mathbb{T}^3)$.
\end{enumerate}
\end{Remark}

It is evident that classical solutions require spatial derivatives of the velocity field $\vu$ to be continuous $\prst$-a.s. This motivates the following definition.

\begin{Definition}[Maximal strong pathwise solution]\label{def:maxsol}
Fix an initial condition, and a complete stochastic basis with a cylindrical Wiener process as in Definition \ref{def:strsol}. Then a triplet $$(\vu,(\tau_L)_{L\in\mn},\mathfrak{t})$$ is said to be a maximal strong pathwise solution to system \eqref{P1} provided

\begin{enumerate}[(a)]
\item $\mathfrak{t}$ is an a.s. strictly positive $(\mathbb{F}_t)$-stopping time;
\item $(\tau_L)_{L\in\mn}$ is an increasing sequence of $(\mathbb{F}_t)$-stopping times such that
$\tau_L<\mathfrak{t}$ on the set $[\mathfrak{t}<T]$,
$\displaystyle{\lim_{L\to\infty}\tau_L}=\mathfrak t$ a.s. and
\begin{equation}\label{eq:blowup}
\sup_{t\in[0,\tau_L]}\|\vu(t)\|_{W^{1,\infty}(\T^3)}\geq L\quad \text{on}\quad [\mathfrak{t}<T] ;
\end{equation}
\item each pair $(\vu,\tau_L)$, $L\in\mn$,  is a local strong pathwise solution in the sense  of Definition \ref{def:strsol}.
\end{enumerate}
\end{Definition}

In view of the above definitions, we are now in a position to state relevant existence theorems. For a proof, we refer to the work by Glatt-Holtz $\&$ Vicol \cite{GHVic}.
\begin{Theorem}[\textbf{Local existence for nonlinear multiplicative noise}]
\label{existence of strong solution}
Let $\StoB$ be a stochastic basis with a complete right-continuous filtration. Suppose that $p\ge 2$, and $m\,\textgreater\,3/p+1$. Let ${W}$ be an $(\mathbb{F}_t) $-cylindrical Wiener process and $\bfu_0$ be a $W^{m,p}(\T^3)$-valued $\mathbb{F}_0$-measurable random variable, and let $\mathcal{G}$ satisfy \eqref{extra}. Then there exists a unique maximal strong pathwise solution
$(\vu,(\tau_L)_{L\in\mn},\mathfrak{t})$ of \eqref{P1} in the sense of Definition \ref{def:maxsol}.
\end{Theorem}



\subsection{Stochastic incompressible Navier--Stokes equations}
	

There is a large and intense literature concerning the incompressible Navier--Stokes equations driven by noise, starting with the work by Flandoli et. al \cite{FlandoliGatarek} where the authors proved the existence of weak martingale solutions to \eqref{P1NS}. As expected, these solutions are weak in both analytical and probabilistic sense. However, to prove existence of dissipative measure-valued solutions for stochastic Euler equations, we first need to introduce the concept of \emph{finite energy weak martingale} solutions to \eqref{P1NS}. Note that, such solutions exist globally in time, and the time-evolution of the energy for such solutions can be controlled in terms of its initial state.

\begin{Definition}[Finite energy weak martingale solution]
\label{def:weakMartin}
Let $\Lambda_{\varepsilon}$ be a Borel probability measure on $L_{\text{div}}^2(\mathbb{T}^3)$. Then $\big[ \big(\Omega_{\varepsilon},\mathbb{F}_{\varepsilon}, (\mathbb{F}_{{\varepsilon},t})_{t\geq0},\mathbb{P}_{\varepsilon} \big);\mathbf{u}_{\varepsilon}, W_{\varepsilon} \big]$ is a weak martingale solution of \eqref{P1NS} if
\begin{enumerate}[(a)]
\item $ \big(\Omega_{\varepsilon},\mathbb{F}_{\varepsilon}, (\mathbb{F}_{{\varepsilon},t})_{t\geq0},\mathbb{P}_{\varepsilon} \big)$ is a stochastic basis with a complete right-continuous filtration,
\item $W_{\varepsilon}$ is a $(\mathbb{F}_{{\varepsilon},t})$-cylindrical Wiener process,
\item the velocity field $\mathbf{u}_{\varepsilon}$ is $L^2_{\dv}(\mathbb{T}^3)$-valued progressively measurable process and $\mathbb{P}-$a.s. 
\begin{align*}\vu (\cdot,\omega)\in C ([0,T];H^{-2\alpha}_{\dv}(\mathbb{T}^3))\cap L^{\infty}(0,T;L^2_{\dv}(\T^3))\cap L^2(0,T;H^1_{\dv}(\T^3))\end{align*}
\item $\Lambda_{\varepsilon}=\mathbb{P}_{\varepsilon}\circ \big[\mathbf{u}_{\varepsilon}(0) \big]^{-1}$,
\item for all $\bm{\varphi}\in  H^{2\alpha}_{\dv}(\mathbb{T}^3)$,
 we have
\begin{equation}\label{eq:energy}
\begin{aligned}
\langle \mathbf{u}_{\varepsilon}(t), \bm\varphi\rangle &= \langle \mathbf{u}_{\varepsilon} (0), \bm{\varphi}\rangle -\int_0^{t}\langle B(\mathbf{u}_{\varepsilon}(s),\mathbf{u}_{\varepsilon}(s)),\bm{\varphi}\rangle\,\mathrm{d}s
+{\varepsilon}\,\int_0^{t} \langle \Delta \mathbf{u}_{\varepsilon}(s)\,,  \bm{\varphi}\rangle\mathrm{d}s 
+\int_0^{t}\langle\mathcal{P}_H\mathcal{G}(\mathbf{u}_{\varepsilon}), \bm{\varphi}\rangle\,\mathrm{d}W
\end{aligned}
\end{equation}
$\mathbb{P}$-a.s. for all $t\in[0,T]$,
\item the energy inequality
\begin{align}
\label{eq:apriorivarepsilon}
&-\int_0^T \partial_t\phi \int_{\T^3}  \frac{1}{2}| {\bf u}_{\varepsilon} |^2\, \D x\,  \D t
+ {\varepsilon}\, \int_0^T \phi \int_{\T^3}|\nabla_x  {\bf u}_{\varepsilon}|^2 \dx\,  {\rm d}t  \\
&\leq \phi(0) \int_{\T^3} \frac{1}{2}{|\bfu_{\varepsilon}(0)|^2}
+\sum_{k=1}^\infty\int_0^T\phi\bigg(\int_{\T^3}\mathcal{P}_H\mathbf{G}_k ({\bf u}_{\varepsilon})\cdot{\bf u}_{\varepsilon}\dx\bigg){\rm d} W_k + \frac{1}{2}\sum_{k = 1}^{\infty}  \int_0^T\phi
\int_{\T^3}|\mathcal{P}_H \mathbf{G}_k ({\bf u}_{\varepsilon}) |^2 \, {\rm d}t \notag
\end{align}
holds $\p$-a.s., for all $\phi\in C_c^\infty([0,T)),\,\phi\,\ge\,0$. 

\end{enumerate}
\end{Definition}

\begin{Remark}
Note that, in view of  Skorohod \cite{sko}, it is possible to consider, 
$$
\Big(\Omega_{\varepsilon}, \mathbb{F}_{\varepsilon},  \mathbb{P}_{\varepsilon} \Big)=\Big([0,1], \mathcal{B}([0,1]), \mathcal{L}_{\R} \Big),
$$
for every $\varepsilon$. Moreover, we may assume the existence of a common Wiener space $W$ for all $\varepsilon$, thanks to a classical compactness argument applied to any chosen subsequence ${\lbrace \varepsilon_n \rbrace}_{n \in \N}$ at once. However, it is worth noticing that, it may not be possible obtain a filtration which is independent of $\varepsilon$, due to lack of pathwise uniqueness for the underlying system.
\end{Remark}
\begin{Remark}
 By using  property \eqref{property}-\eqref{property1}, we can recast the item $(e)$ in Definition \ref{def:weakMartin} as
\begin{enumerate}
	\item[$(e')$] For all $\bm{\varphi}\in C_{\text{div}}^\infty(\mathbb{T}^3)$,
\begin{equation}\label{e1q:energy1}
	\begin{aligned}
		\langle \mathbf{u}_{\varepsilon}(t), \bm\varphi\rangle = \langle \mathbf{u}_{\varepsilon} (0), \bm{\varphi}\rangle &+\int_0^{t}\langle \mathbf{u}_{\varepsilon}\otimes\mathbf{u}_{\varepsilon}(s),\nabla_x\bm{\varphi}\rangle\,\mathrm{d}s \\&
		-{\varepsilon}\,\int_0^{t} \langle \nabla_x \mathbf{u}_{\varepsilon}(s)\,,  \nabla_x\bm{\varphi}\rangle\mathrm{d}s 
		+\int_0^{t}\langle\mathcal{G}(\mathbf{u}_{\varepsilon}), \bm{\varphi}\rangle\,\mathrm{d}W
	\end{aligned}
\end{equation}
\end{enumerate}
holds $\mathbb{P}$-a.s., for all $t\in[0,T]$.
\end{Remark}

Regarding the existence of finite energy martingale solutions, one may follow the arguments given by Flandoli et. al. \cite{FlandoliGatarek} to obtain the following result.
\begin{Theorem}[\textbf{Existence of martingale solution for Naiver Stokes}]
\label{thm:Romeo1}
 Assume that $\Lambda_{\varepsilon}$ is a Borel probability measure on $L_{\rm {div}}^2(\mathbb{T}^3)$ such that the following moment estimate
$$\int_{L_{\dv}^2(\mathbb{T}^3)} {\| \mathbf{u} \|}^p_{L_{{\dv}}^2(\T^3)}\, \mathrm{d}\Lambda_{\varepsilon}(\mathbf{u})<\infty,$$
holds for all $1\leq p<\infty$. Moreover, assume that \eqref{FG1} and \eqref{FG2} hold. Then there exists a finite energy weak martingale solution of \eqref{P1NS} in the sense of Definition \ref{def:weakMartin} with initial law $\Lambda_{\varepsilon}$.
\end{Theorem}

\begin{proof}
Existence proof for a weak martingale solution follows from the work of Flandoli et. al. \cite{FlandoliGatarek}. To prove the energy inequality, one may simply apply It\^o formula to obtain item (f) of Definition \ref{def:weakMartin}. Indeed, this is very similar to the recent works on compressible fluids, see Breit et. al. \cite{BrFeHobook}. The details are left to the interested reader.
\end{proof}

\begin{Remark}
(\textbf{A different form of energy inequality})
It is well-known that, one can establish weak-strong uniqueness principle only in the class of dissipative weak solutions, i.e., weak solutions satisfying an appropriate energy inequality. To that context, we make use of a different form of energy inequality for the proof of weak-strong uniqueness related to incompressible Navier--Stokes equations. 
First observe that, in view of a standard cut-off argument applied to \eqref{eq:apriorivarepsilon}, energy inequality holds for a.e. $0\le\,s\,\textless\,t \,\in(0,T):$
\begin{align}
&\int_{\T^3} \frac{1}{2}| {\bf u}_{\varepsilon}(t) |^2 \,\D x 
+ {\varepsilon}\, \int_s^t \int_{\T^3}|\nabla_x  {\bf u}_{\varepsilon}|^2 \dx\,  {\rm d}s \label{eq:apriorivarepsilon12} \\
&\quad \leq \int_{\T^3} \frac{1}{2}{|\bfu_{\varepsilon}(s)|^2} \, \D x
+\sum_{k=1}^\infty\int_s^t\bigg(\int_{\T^3}\mathcal{P}_H\mathbf{G}_k ({\bf u}_{\varepsilon})\cdot{\bf u}_{\varepsilon}\dx\bigg){\rm d} W_k + \frac{1}{2}\sum_{k = 1}^{\infty}  \int_s^t\int_{\T^3}|\mathcal{P}_H \mathbf{G}_k ({\bf u}_{\varepsilon}) |^2 \, {\rm d}s. \notag
\end{align}
It follows from \eqref{eq:apriorivarepsilon12} that the limits
$$\esslim_{\tau\to s^+}\int_{\mathbb{T}^3}\frac{1}{2}|u(\tau)|^2\, \D x,\,\,\,\esslim_{\tau\to t^-}\int_{\mathbb{T}^3}\frac{1}{2}|u(\tau)|^2\, \D x$$
exist $\mathbb{P}$-a.s. for a.a.\,\,$0\,\le\,s\,\le\,t\,\le\,T$ including $s=0$.
Finally, in view to the weak lower-semicontinuity of convex functionals, we have for any $t\in[0,T)\,\,\mathbb{P}$-a.s.
$$\liminf_{\tau\to t^-}\int_{\mathbb{T}^3}\frac{1}{2}|u(\tau)|^2\, \D x\,\ge\,\int_{\mathbb{T}^3}\frac{1}{2}|u(t)|^2\, \D x$$
By making use of the above informations, relative energy inequality \eqref{eq:apriorivarepsilon12} can be rewritten as 
\begin{align}\label{eq:apriorivarepsilon1}
\begin{aligned}
& \int_{\T^3} \frac{1}{2}| {\bf u}_{\varepsilon}(t) |^2  \, \D x
+ {\varepsilon}\, \int_0^t \int_{\T^3}|\nabla_x  {\bf u}_{\varepsilon}|^2 \dx\,  {\rm d}s\\
& \quad \leq \int_{\T^3} \frac{1}{2}{|\bfu_{\varepsilon}(0)|^2}\, \D x
+\sum_{k=1}^\infty\int_0^t\bigg(\int_{\T^3}\mathcal{P}_H\mathbf{G}_k ({\bf u}_{\varepsilon})\cdot{\bf u}_{\varepsilon}\dx\bigg){\rm d} W_k
+ \frac{1}{2}\sum_{k = 1}^{\infty}  \int_0^t
\int_{\T^3}|\mathcal{P}_H \mathbf{G}_k ({\bf u}_{\varepsilon}) |^2 \, {\rm d}s
\end{aligned}
\end{align}
holds $\mathbb P$-a.s., for all $t\in[0,T]$.
\end{Remark}

\subsection{Measure-valued martingale solutions}
In general, in view of the energy inequality \eqref{eq:apriorivarepsilon1}, solutions to incompressible Navier--Stokes equations only has a uniform energy bound, usually $L^2(\T^3)$. However, such {\it a priori} bound does not guarantee weak convergence of nonlinear terms $\vu \otimes \vu, \textbf{G}^2(\vu) \in L^1(\T^3)$, due to the presence of oscillations and concentration effects. In this scenario, one can only identify weak limits (corresponding to nonlinear terms) as a combination of Young measure and concentration measure. 

Note that the Young measures, which are probability measures on the phase space, capture oscillations in the solution. On the other hand concentration defect measures, which are measures on physical space-time, accounts for blow up type collapse due to possible concentration points. In what follows, we use two different forms of concentration defect measures. To illustrate the difference, we consider the following situation:
\begin{itemize}
\item Let $v_{\varepsilon}$ converges weakly in $L^2(\T^3)$ to a function $v$, and we assume that $\int_{\T^3} |v_{\varepsilon}|^2 \,dx \le C$.
\end{itemize}
Inspired by Banach-Alaoglu theorem, one can define a defect measure - which is a non-negative Radon measure, as 
$$
\mu_1 := \mbox{Weak -*}\lim_{\varepsilon \rightarrow 0} \Big(  |v_{\varepsilon}|^2 - |v|^2 \Big) \in \mathcal{M}^{+}_b(\Pi_T).
$$
In a similar fashion, by making use of Young measure theory, one can define another defect measure $\mu_2$ as
$$
\mu_2 := \mbox{Weak -*}\lim_{\varepsilon \rightarrow 0} \Big(  |v_{\varepsilon}|^2 - \langle \mathcal{V}_{t,x}(\lambda); |\lambda|^2 \rangle \Big) \in \mathcal{M}^{+}_b(\Pi_T).
$$
It is well-known that $\mu_1$ is too large to describe concentration effects in a useful way, while the main advantage of the concentration defect measure $\mu_2$ is that it allows to describe weak limits in terms of Young measure. It is easy to see that, thanks to H\"{o}lder inequality, $\mu_2 \le \mu_1$. We shall make use of both forms of concentration defect measures below to define a notion of measure valued solution.


\subsubsection{Dissipative measure-valued martingale solutions }
	
Keeping in mind the previous discussion, we are ready to introduce the concept of \textit{dissipative measure--valued martingale solution} to the stochastic compressible Euler system. In what follows, let $\mathcal{S} =  \R^3$ be the phase space associated to the incompressible Euler system. 

\begin{Definition}[Dissipative measure-valued martingale solution]
\label{def:dissMartin}
Let $\Lambda$ be a Borel probability measure on $L_{\text{div}}^2(\mathbb{T}^3)$. Then $\big[ \big(\Omega,\mathbb{F}, (\mathbb{F}_{t})_{t\geq0},\mathbb{P} \big); \mathcal{V}^{\omega}_{t,x}, W \big]$ is a dissipative measure-valued martingale solution of \eqref{P1}, with initial condition $\mathcal{V}^{\omega}_{0,x}$; if
\begin{enumerate}[(a)]
\item $\mathcal{V}^{\omega}$ is a random variable taking values in the space of Young measures on $L^{\infty}_{w^*}\big([0,T] \times \T^3; \mathcal{P}\big(\mathcal{S})\big)$. In other words, $\p$-a.s.
$\mathcal{V}^{\omega}_{t,x}: (t,x) \in [0,T] \times \T^3  \rightarrow \mathcal{P}(\mathcal{S})$ is a parametrized family of probability measures on $\mathcal{S}$,
\item $ \big(\Omega,\mathbb{F}, (\mathbb{F}_{t})_{t\geq0},\mathbb{P} \big)$ is a stochastic basis with a complete right-continuous filtration,
\item $W$ is a $(\mathbb{F}_{t})$-cylindrical Wiener process,
\item the average velocity $\langle \mathcal{V}^{\omega}_{t,x}; \vu  \rangle$ satisfies,  for any $\bm{\varphi}\in C_{\text{div}}^\infty(\mathbb{T}^3)$,  $t\mapsto \langle \langle \mathcal{V}^{\omega}_{t,x}; \vu  \rangle (t, \cdot),\bm\varphi\rangle\in C[0,T]$, $\mathbb{P}$-a.s., the function $t\mapsto \langle \langle \mathcal{V}^{\omega}_{t,x}; \vu  \rangle (t, \cdot),\bm{\phi}\rangle$ is progressively measurable, and for any $\bm\varphi\in C^1(\mathbb{T}^3)$,
\begin{align*}
	\int_{\mathbb{T}^3}\langle \mathcal{V}^{\omega}_{t,x}; \vu  \rangle\cdot\nabla_x \bm{\varphi}\,dx=0
\end{align*}
for almost $t\in[0,T]$, $\mathbb{P}-$a.s., and 
\begin{align*}
\mathbb{E}\, \bigg[ \sup_{t\in(0,T)}\Vert  \langle \mathcal{V}^{\omega}_{t,x}; \vu  \rangle (t, \cdot) \Vert_{L_{\dv}^2(\mt)}^p\bigg]<\infty 
\end{align*}
for all $1\leq p<\infty$,
\item $\Lambda=\mathcal{L}[\mathcal{V}^{\omega}_{0,x}]$,
\item there exists a $H_{\text{div}}^{-1}(\T^3)$-valued square integrable continuous martingale $\mathcal{M}^1_{E}$, such that the integral identity 
\begin{equation} \label{second condition measure-valued solution}
\begin{aligned}
&\int_{\T^3} \langle \mathcal{V}^{\omega}_{\tau,x};\vu \rangle \cdot \bm{\varphi}(\tau, \cdot) \,\D x - \int_{\T^3} \langle \mathcal{V}^{\omega}_{0,x};\vu \rangle \cdot \bm{\varphi}(0,\cdot) \, \D x \\
&\qquad = \int_{0}^{\tau} \int_{\T^3}\langle \mathcal{V}^{\omega}_{t,x}; {\vu \otimes \vu }\rangle: \nabla_x \bm{\varphi} \, \D x \, \D t + \int_{\T^3} \bm{\varphi}\,\int_0^{\tau} d\mathcal{M}^1_{E}(t) \,\D x+ \int_{0}^{\tau} \int_{\T^3}  \nabla_x \bm{\varphi}: d\mu_C,
\end{aligned}
\end{equation}
holds  $\p$-a.s., for all $\tau \in [0,T)$, and for all $\bm{\varphi} \in C_{\text{div}}^{\infty}(\T^3;\mathbb{R}^3)$, where $\mu_C\in L_{w*}^\infty([0,T];\mathcal{M}_b({\T^3}))$, $\p$-a.s., is a tensor--valued measure; $\mu_C$ is called \textit{concentration defect measures};
\item there exists a real-valued square integrable continuous martingale $\mathcal{M}^2_{E}$, such that the following inequality
\begin{align} \label{third condition measure-valued solution}
\begin{aligned}
 \mathrm{E}(t+)\, \leq\, &\mathrm{E}(s-)+ \frac{1}{2} \sum_{k\ge\,1}\int_s^{t} \int_{\T^3} \left\langle \mathcal{V}^{\omega}_{\tau,x};{|{\textbf G_k( u)}|^2} \right\rangle \,\, \D x\, {\rm d}\tau \\
 &\quad  -\frac{1}{2}\sum_{k\ge\,1}\int_s^t\int_{\mathbb{T}^3}\Big(\mathcal{Q}_H\left\langle \mathcal{V}^{\omega}_{\tau,x};{|{\textbf G_k( u)}|}\right\rangle\Big)^2\,\D x\,{\rm d}\tau + \frac12\int_s^{t} \int_{\T^3} d \mu_D + \int_s^{t}  d\mathcal{M}^2_{E},
\end{aligned}			
\end{align}
holds  $\p$-a.s., for all $0\,\le\,s\,\textless\,t\in (0,T)$ with
$$\mathrm{E}(t-):=\lim_{r\to 0^+}\frac{1}{r}\int_{t-r}^t\bigg(\int_{\mathbb{T}^3}\left\langle \mathcal{V}^{\omega}_{s,x};\frac{|{\bf u}|^2}{2} \right\rangle \dx +\mathcal{D}(s)\bigg)\,{\rm d}s$$
$$\mathrm{E}(t+):=\lim_{r\to 0^+}\frac{1}{r}\int_t^{t+r}\bigg(\int_{\mathbb{T}^3}\left\langle \mathcal{V}^{\omega}_{s,x};\frac{|{\bf u}|^2}{2} \right\rangle \dx +\mathcal{D}(s)\bigg)\,{\rm d}s$$
Here $\mu_D \in L_{w*}^\infty([0,T];\mathcal{M}_b({\T^3}))$, $\p$-a.s., $\mathcal{D}\in L^{\infty}(0,T)$, $\mathcal{D}\geq 0$, $\p$-almost surely, and $\E \big[ \esssup_{t \in (0,T)}\mathcal{D}(t)\big] < \infty$, with initial energy
$$
\mathrm{E}(0-)= 
\int_{\mathbb{T}^3} \frac{1}{2} |\textbf{u}_0|^2\,\dx.
$$
\item there exists a constant $C>0$ such that
\begin{equation} \label{fourth condition measure-valued solutions}
	 \int_{0}^{\tau} \int_{\T^3} d|\mu_C| + \int_{0}^{\tau} \int_{\T^3} d|\mu_D| \leq C \int_{0}^{\tau} \mathcal{D}(t) dt,
\end{equation}	
$\p$-a.s., for every $\tau \in (0,T)$.
\item For any given stochastic process $h(t)$, adapted to $(\mathbb{F}_{t})_{t\geq0}$, given by
$$
\mathrm{d}h  = D_t^dh\,\mathrm{d}t  + \mathbb{D}_t^s h\,\mathrm{d}W,
$$
satisfying
\begin{equation*} 
h \in C([0,T]; W^{1,q} \cap C (\tor)), \quad 
\E\bigg[ \sup_{t \in [0,T]} \| h \|_{W^{1,q}}^2 \bigg]^q < \infty, \quad \text{$\mathbb{P}$-a.s. for all }\ 1 \leq q < \infty,
\end{equation*}
with 
\begin{align*}
&D_t^d h \in L^q(\Omega;L^q(0,T;W^{1,q}(\mt))),\quad
\mathbb D_t^s h \in L^2(\Omega;L^2(0,T;L_2(\mathfrak U;L_{\dv}^2(\tor)))),\nonumber\\
&\bigg(\sum_{k\geq 1}|\mathbb D_t^s h(e_k)|^q\bigg)^\frac{1}{q} \in L^q(\Omega;L^q(0,T;L^{q}(\mt))),
\end{align*}
the cross variation between $h$ and the square integrable continuous martingale $M^1_E$ is given by 
\begin{align*}
\Big<\hspace{-0.14cm}\Big<h(t),  M^1_{E}(t)  \Big>\hspace{-0.14cm}\Big> 
= \sum_{i,j}\Bigg(\sum_{k = 1}^{\infty}  \int_0^t\langle \mathcal{P}_H \left\langle \mathcal{V}^{\omega}_{s,x}; \mathbf{G}_k (\vu )\right\rangle,g_i\rangle \,\langle\mathbb{D}_t^s h(e_k), g_j\rangle \,ds\Bigg) g_i\otimes g_j .
\end{align*}
Here $g_i$'s are orthonormal basis for $H^{-1}_{\dv}(\mathbb{T}^3)$ and bracket $\langle\cdot,\cdot\rangle$ denotes inner product in the same space.
\end{enumerate}
\end{Definition}

\begin{Remark}
Notice that, a standard Lebesgue point argument applied to \eqref{third condition measure-valued solution} reveals that the energy inequality holds for a.e. $0 \le s <t$ in $(0,T)$:
\begin{align}
\label{energy_001}
& \int_{\mathbb{T}^3}\left\langle \mathcal{V}^{\omega}_{t,x};\frac{|{\bf u}|^2}{2} \right\rangle \dx +\mathcal{D}(t) \, \leq\, \int_{\mathbb{T}^3}\left\langle \mathcal{V}^{\omega}_{s,x};\frac{|{\bf u}|^2}{2} \right\rangle \dx +\mathcal{D}(s) + \frac{1}{2} \sum_{k\ge\,1}\int_s^{t} \int_{\T^3} \left\langle \mathcal{V}^{\omega}_{s,x};{|{\textbf G_k( u)}|^2} \right\rangle \,\, \D x\, {\rm d}\tau \notag \\ &\quad  -\frac{1}{2}\sum_{k\ge\,1}\int_s^t\int_{\mathbb{T}^3}\Big(\mathcal{Q}_H\left\langle \mathcal{V}^{\omega}_{s,x};{|{\textbf G_k( u)}|}\right\rangle\Big)^2\,\D x\,{\rm d}\tau + \frac12\int_s^{t} \int_{\T^3} d \mu_D + \int_s^{t}  d\mathcal{M}^2_{E}, \,\,\p-a.s.
\end{align}
However, as it is evident from Section~\ref{proof2}, we require energy inequality to hold for \emph{all} $s, t \in (0,T)$ to demonstrate weak-strong uniqueness principle. 
\end{Remark}


%
%

\subsection{Statements of main results}

We now state main results of this paper. To begin with, regarding the existence of dissipative measure-valued martingale solutions, we have the following theorem.

\begin{Theorem} [\textbf{Existence of Measure-Valued Solution}]
	\label{thm:exist}
Assume $\vc{G}_k$ satisfies \eqref{FG1}, \eqref{FG2}, and $ {\bf u}_{\varepsilon}$ be a family of finite energy weak martingale solutions to the stochastic incompressible Navier--Stokes system \eqref{P1NS}. Let the corresponding initial data ${\bf u}_{0}$, the initial law $\Lambda$, given on the space $L_{\dv}^2(\T^3)$, be independent of $\varepsilon$, satisfying the following moment estimate
\begin{align}\label{initial}
\int_{L_{\dv}^2} \Vert  \mathbf{q}\Vert^p_{L_{\dv}^2(\T^3)}\, \mathrm{d}\Lambda (\mathbf{q})<\infty,
\end{align}
holds for all $1\leq p<\infty$. 
Then the family ${\lbrace {\bf u}_{\varepsilon}  \rbrace}_{\varepsilon>0}$
generates, as $\varepsilon \mapsto 0$, a Young measure ${\lbrace \mathcal{V}^{\omega}_{t,x} \rbrace}_{t\in [0,T]; x\in \T^3}$ which is a dissipative measure-valued martingale solution to
the stochastic incompressible Euler system (\ref{P1}), in the sense of Definition~\ref{def:dissMartin}, with initial data $\mathcal{V}^{\omega}_{0,x}=\delta_{ {\bf u}_0(x)}$ almost surely.
\end{Theorem}

We then establish the following weak (measure-valued)-strong uniqueness principle:
\begin{Theorem}[\bf Weak-Strong Uniqueness For Nonlinear Noise] \label{Weak-Strong Uniqueness}
Let $\big[ \big(\Omega,\mathbb{F}, (\mathbb{F}_{t})_{t\geq0},\mathbb{P} \big); \mathcal{V}^{\omega}_{t,x}, W \big]$ be a dissipative measure-valued martingale solution to the system \eqref{P1}. On the same stochastic basis $\big(\Omega,\mathbb{F}, (\mathbb{F}_{t})_{t\geq0},\mathbb{P} \big)$, let us consider the unique maximal strong pathwise solution to the Euler system (\ref{P1}), driven by the same cylindrical Wiener process $W$, given by 
$(\bar{{\bf u}},(\tau_L)_{L\in\mn},\mathfrak{t})$ with the initial data $\bar{\bf u}(0)$ satisfies 
\begin{equation*} 
\mathcal{V}^{\omega}_{0,x}= \delta_{\bar{\bf u}(0,x)}, \,\p-\mbox{a.s.,}\, \mbox{for a.e. } x \in \T^3.
\end{equation*}
Then for $L \in \N$, a.e. $t\in[0,T]$, $\mathcal{D}(t \wedge \tau_L)=0$, $\p$-a.s., and  for a.e. $t\in[0,T]$, $\mathbb{P}$-a.s.
\begin{equation*}
\mathcal{V}^{\omega}_{t \wedge \tau_L,x}	= \delta_{\bar{\bf u}(t \wedge \tau_L,x)}, \,  \mbox{for a.e. }x\in\T^3.
\end{equation*}
\end{Theorem}

%
%

Next, we move our attention to the stochastic incompressible Navier--Stokes equations given by \eqref{P1NS} (with $\varepsilon=1$). For the deterministic counterpart of \eqref{P1NS}, Prodi \cite{prodi} and Serrin \cite{serrin} established weak-strong uniqueness principle under additional regularity of a solution $\mathbf{U}\in L^r([0,T],L^s(\mathbb{T}^3))$, where $r$ and $s$ satisfies the relation $2/r + 3/s =1$, with $s \in (3, \infty)$. In the stochastic setup, uniqueness of finite energy weak martingale solutions for \eqref{P1NS} seems to be out of reach. However, one can obtain a conditional uniqueness result like Prodi and Serrin. In fact, our  objective is to give a sufficient condition under which martingale solutions of Naiver Stokes equation \eqref{P1NS} is unique in class of finite energy weak martingale solutions. Note that, in comparison to deterministic analysis, we need additional continuity assumption on stochastic solutions to deal with the stopping time.

\begin{Theorem}[\textbf{Weak-Strong Uniqueness For Naiver Stokes}] 
\label{Weak-Strong Uniqueness1}
Let $\bf u, U$ be two finite energy weak martingale solutions to the system \eqref{P1NS}, defined on the same stochastic basis $\big(\Omega,\mathbb{F}, (\mathbb{F}_{t})_{t\geq0}, W,\mathbb{P} \big)$ with same intial data. Let the solution $\bf U$ additionally satisfies
\begin{equation} 
\label{cond}
\E\bigg[ \sup_{t \in [0,T]} \|{\bf U} \|_{L^{s}(\mathbb{T}^3)} \bigg] < \infty, \,\, \text{for some}\,\, 3\textless\,s\,\textless\,\infty.
\end{equation}
Then $\mathbb{P}$-almost surely, for all $t\in[0,T]$ 
\begin{equation*}
{\vu }({t,x})	= {{\bf U}(t,x)},\,  \mbox{for a.e. }x\in \T^3.
\end{equation*}
\end{Theorem}


\section{Proof of Theorem~\ref{thm:exist}}
\label{proof1}

The proof of existence is essentially based on compactness method. For the compactness argument in space and time variables we make use of Young measure theory, while for the compactness argument in probability variable we rely on Skorokhod representation theorem. However, as alluded to before, our path spaces are not Polish spaces (which is required for classical Skorokhod theorem), therefore we rely on Skorokhod-Jakubowski theorem (\cite[Corollary 2]{BrzezniakHausenblasRazafimandimby})which is taylor made to deal with so called quasi-Polish spaces. As usual, we obtain the convergence of the approximate sequence on another probability space and the existence of dissipative measure-valued martingale solution follows, thanks to Young measure theory. Observe that the existence of a pathwise solution (typically obtained by  Gy\"{o}ngy-Krylov's characterization of convergence in probability) seems not possible due to the lack of pathwise uniqueness for the underlying system.
\subsection{A priori Bounds}
Recall that, the existence of finite energy weak martingale solution of stochastic incompressible Navier--Stokes system (\ref{P1NS})
$$\big[ \big(\Omega,\mathbb{F}, (\mathbb{F}_{{\varepsilon},t})_{t\geq0},\mathbb{P} \big);\mathbf{u}_{\varepsilon}, W\big]$$
is well established, thanks to the Theorem~\ref{thm:Romeo1}. Observe that the filtration $(\mathbb{F}_{{\varepsilon},t})_{t\geq0}$ depends on $\varepsilon$, and lack of pathwise uniqueness for (\ref{P1NS}) does not allow us to choose the filtration independent of $\varepsilon$. Having said this, however, note that the Brownian motion and the probability space can be chosen indepedent of $\varepsilon$.

To obatin a-priori estimate for the approximate solution, we make use of the energy inequality \eqref{eq:apriorivarepsilon1} to obatin for all $t\in[0,T]$, and $1\,\le\,p\,\textless\,\infty$
$$
\mathbb{E} \Big[\|{\bf u}_\varepsilon(t)\|_{L^2(\mathbb{T}^3)}^2 \Big]^p\le\,\mathbb{E} \Big[\|{\bf u}_0\|_{L^2(\mathbb{T}^3)}^2 \Big]^p+C\int_0^t\Bigg(1+\mathbb{E}\Big[\|{\bf u}_\varepsilon(s)\|_{L^2(\mathbb{T}^3)}^2\Big]^p\Bigg)ds.
$$
Therefore, a simple application Gronwall lemma yield
$$\mathbb{E} \Big[\|\vu _\varepsilon\|_{L^2(\mathbb{T}^3)}^2 \Big]^p\le\,C\,\Bigg(1+\mathbb{E} \Big[\|\vu _0\|_{L^2(\mathbb{T}^3)}^2\Big]^p\Bigg)$$
Again by energy inequality \eqref{eq:apriorivarepsilon1}, we have for $1\,\le\,p\,\textless\,\infty$
\begin{align*}
	\mathbb{E}\bigg[&\sup_{t\in[0,T]}\|{\bf u}_{\varepsilon}\|_{L^2(\mathbb{T}^3)}^2+\varepsilon\int_0^T\|\nabla_x{\bf u}_{\varepsilon}\|_{L^2(\mathbb{T}^3)}^2\, \D t\bigg]^p\,\\
	&\qquad \qquad \le\,C\int_{L_{\text{div}}^2(\mathbb{T}^3)} \Vert  \mathbf{q}\Vert^{2p}_{L_{\text{div}}^2(\T^3)}\, \mathrm{d}\Lambda (\mathbf{q})+ C\,\mathbb{E}\bigg[\sup_{t\in[0,T]}{\int_{\mathbb{T}^3}}\int_0^t {\bf u}_\varepsilon.\mathcal{P}_H\mathcal{G}({\bf u}_\varepsilon) \,\D W \,\D x	\bigg]^p.
\end{align*}
To handle the right most term of the above inequality, we make use of the classical Burkholder-Davis-Gundi (BDG) inequality to obtain
\begin{align*}
	\mathbb{E} & \bigg[\sup_{t\in[0,T]}  {\int_{\mathbb{T}^3}}\int_0^t {\bf u}_\varepsilon.\mathcal{P}_H\mathcal{G}({\bf u}_\varepsilon) \, \D W \, \D x	\bigg]^p\,\le\,C\,\mathbb{E}\bigg[\int_0^T\sum_{k\,\ge\,1}\bigg(\int_{\mathbb{T}^3}\mathcal{P}_H\textbf {G}_k(\vu _\varepsilon)\vu _{\varepsilon}\, \D x\bigg)^2\bigg]^{p/2}\\
	&\le\,C\,\mathbb{E}\bigg[\int_0^T\|{\bf u}_\varepsilon(t)\|_{L^2(\mathbb{T}^3)}^2\,\sum_{k\,\ge\,1}\|{\textbf G}_k({\bf u}_{\varepsilon}(t))\|_{L^2(\mathbb{T}^3)}^2 \, \D t\bigg]^{p/2}\\
	&\le\, C\,\mathbb{E}\bigg[\int_0^T\|{\bf u}_\varepsilon(t)\|_{L^2(\mathbb{T}^3)}^2\,(1+\|{\bf u}_{\varepsilon}(t)\|_{L^2(\mathbb{T}^3)}^2) \, \D t\bigg]^{p/2}\\
	&\le\,C\mathbb{E}\bigg[\int_0^T\|{\bf u}_\varepsilon(t)\|_{L^2(\mathbb{T}^3)}^2\, \D t \bigg]^p+\mathbb{E}\bigg[\int_0^T(1+\|{\bf u}_\varepsilon(t)\|_{L^2(\mathbb{T}^3)}^2)\, \D t\bigg]^p
	\le\,C\Bigg(1+\int_{L_{\text{div}}^2(\mathbb{T}^3)} \hspace{-0.2cm} \Vert  \mathbf{q}\Vert^{2p}_{L_{\text{div}}^2(\T^3)}\, \mathrm{d}\Lambda (\mathbf{q})\Bigg).
\end{align*}
This implies that, for any $1 \le p<\infty$, we have
\begin{align*}
\stred\bigg[\sup_{0\leq t\leq T}\int_{\mt}{|{\bf u}_\varepsilon|^2}\,\dif x
+ \varepsilon \int_0^T\int_{\mt}|\nabla_x\bu_\varepsilon|^2\dif x\,\dif s\bigg]^p  \leq \,C\Bigg(1+\int_{L_{\text{div}}^2(\T^3)} \Vert  \mathbf{q}\Vert^{2p}_{L_{\text{div}}^2(\T^3)}\, \mathrm{d}\Lambda (\mathbf{q})\Bigg)\le\,C(p,\Lambda,T)
\end{align*}
Above relation leads to the following uniform bound
\begin{align}
 \bfu_\varepsilon&\in L^{p}(\Omega;L^\infty(0,T;L_{\text{div}}^{2}( \mathbb T^3))).\label{apv}
\end{align}

\subsection{Tightness and almost sure representations}
\label{subsec:compactness}

For our purpose, to secure almost sure convergence in the probability variable ($\omega$-variable) we make use of Skorokhod-Jakubowski version \cite{BrzezniakHausenblasRazafimandimby,Jakubowski} of the classical Skorokhod representation theorem. It is well-known that such a result can be obtained by establishing tightness of probability measures related to the random variables in quasi-Polish spaces. In what follows, our first aim is to establish the tightness of the probabilty measures (laws) generated by the approximate solutions. To do so, we first introduce the following path space $\mathcal{Y}$ for these measures:
\begin{align*}
\mathcal{Y}_\bu&= C_w([0,T];L_{\text{div}}^2(\mt)),&\mathcal{Y}_W&=C([0,T];\mathfrak{U}_0),\\
\mathcal{Y}_{C} &= \big(L^{\infty}(0,T; \mathcal{M}_b(\T^3)), w^* \big),&\mathcal{Y}_{E} &= \big(L^{\infty}(0,T; \mathcal{M}_b(\T^3)), w^* \big),\\
\mathcal{Y}_{D}&= \big(L^{\infty}(0,T; \mathcal{M}_b(\T^3)), w^* \big) &\mathcal{Y}_{\mathcal{V}} &= \big(L^{\infty}((0,T)\times \T^3; \mathcal{P}(\R^3)), w^* \big),\\
\mathcal{Y}_X&=C([0,T];H_{\text{div}}^{-1}(\mathbb{T}^3)), & \mathcal{Y}_Y&=C([0,T]; \R),\\
\mathcal{Y}_F&=\big(L^{\infty}(0,T; \mathcal{M}_b(\T^3)), w^* \big),
\end{align*}
Let us denote by $\mu_{\bu_\varepsilon}$, and $\mu_{W_\varepsilon}$ respectively, the law of $\bu_\varepsilon$, and $W_\varepsilon$ on the corresponding path space. Moreover, for the martingale terms, let $\mu_{X_\varepsilon}$, and $\mu_{Y_\varepsilon}$ denote the law $X_\varepsilon:=\int_0^t {\mathcal{P}_H\mathcal{G}}(\vu _\varepsilon )\,\Dif W$, and $Y_\varepsilon:=\int_0^t\intTor{ \vu_{\varepsilon} \cdot {\mathcal{P}_H\mathcal {G}_k}( \vu _\varepsilon ) } \,\Dif W$ on the corresponding path spaces respectively. Furthermore, let $\mu_{C_\varepsilon}$, $\mu_{D_\varepsilon}$, $\mu_{E_\varepsilon}$, and $\mu_{{\mathcal{V}}_\varepsilon}$ denote the law of 
\begin{align*}
& C_\varepsilon:= {\vu _\varepsilon\otimes \vu _\varepsilon}, \quad D_{\varepsilon}:= \sum_{k\geq1}| \vc{G}_k({\bf u}_{\varepsilon}) |^2,\quad E_\varepsilon :={\frac{1}{2}|{\bf u}_\varepsilon|^2}, \quad \quad {\mathcal{V}}_\varepsilon := \delta_{ \vu _\varepsilon}, \quad F_\varepsilon:=\frac{1}{2}\sum_{k\,\ge\,1}|\mathcal{Q}_h{\vc G_k(u_\varepsilon)}|^2,
\end{align*}
respectively, on the corresponding path spaces. Finally, we denote by $\mu^\varepsilon$, the joint law of all the variables on $\mathcal{Y}$. As stated before, our aim is now to establish tightness of $\{\mu^\varepsilon;\,\varepsilon\in(0,1)\}$. To this end, first note that tightness of $\mu_{W_\ep}$ is straightforward. Therefore, we focus on proving tightness of other variables.

\begin{Proposition}\label{prop:rhotight}
The set $\{\mu_{\bu_\varepsilon};\,\varepsilon\in(0,1)\}$ is tight on $\mathcal{Y}_\bu$. 
\end{Proposition}

\begin{proof}
For convenience, we rewrite the equation \eqref{e1q:energy1} as
\begin{align*}
	\int_{\mathbb{T}^3}\vu _{\varepsilon}(t)\cdot \bm{\varphi}\, \D x=\int_{\mathbb{T}^3}\vu_{\varepsilon}(0)\cdot\bm{\varphi}\, \D x +\int_0^t\int_{\mathbb{T}^3}I_{\varepsilon}(s):\nabla\bm{\varphi} \, \D x\, \D s+\int_0^t\int_{\mathbb{T}^3}\mathcal{P}_H\mathcal G (\vu_{\varepsilon})\cdot\bm{\varphi} \, \D x\, \D W
\end{align*}
for all $t\in[0,T]$, for all $\bm{\varphi}\in C_{\text{div}}^{\infty}(\mathbb{T}^3)$. 
where $$I_{\varepsilon}:=-\varepsilon\nabla\vu_{\varepsilon}+\vu_{\varepsilon}\otimes\vu_{\varepsilon}$$
From the a priori estimate in \eqref{apv}, we obtain
$$I_{\varepsilon}\in L^1(\Omega;L^2(0,T;L^1(\mathbb{T}^3)))\subset L^1(\Omega;L^2(0,T;W^{-2,2}(\mathbb{T}^3)))$$
uniformly in $\varepsilon$. Let us consider the functional
$$\langle\mathcal{I}_{\varepsilon}({t),\bm{\varphi}}\rangle:=\int_0^t\int_{\mathbb{T}^3} I_{\varepsilon}(s):\nabla\bm{\varphi} \,\D x\, \D s,$$
which is related to the deterministic part of equation. Then we deduce from above the following estimate
$$\mathbb{E}\bigg[\|\mathcal{I}_{\varepsilon}\|_{W^{1,2}(0,T;W_{\text{div}}^{-3,2}(\mathbb{T}^3))}\bigg]\le\,C(T).$$
For the stochastic term, we have, for $a\,\textgreater\,2$
\begin{align*}
\mathbb{E} &\Bigg[\bigg\|\int_{t}^s \mathcal{P}_H \mathcal G(\vu _{\varepsilon}) dW\bigg\|_{L^2(\mathbb{T}^3)}^a \Bigg]\\ & \le\,C\,\mathbb{E}\Bigg[\bigg(\int_t^s\|\mathcal{P}_H\mathcal{G}(\vu _{\varepsilon})\|_{L_2(\mathfrak{U},L^2(\mathbb{T}^3))}^2d\sigma\bigg)^{a/2} \Bigg]
\le\,C\,\mathbb{E} \Bigg[\bigg(\int_{t}^s(1+\|\vu _{\varepsilon}(\sigma)\|_{L^2(\mathbb{T}^3)}^2)d\sigma\bigg)^{a/2} \Bigg]\\
&\le\,C\,\bigg(|t-s|^{a/2}\bigg(1+\mathbb{E}\Big[\sup_{t\in[0,T]}\|\vu _{\varepsilon}(t)\|_{L^2(\mathbb{T}^3)}\Big]^{a/2}\bigg)\bigg)
\le\,C\,|t-s|^{a/2}\Big(1+\mathbb{E}\Big[\|\vu _{\varepsilon}(0)\|_{L^2(\mathbb{T}^3)}^a\Big]\Big).
\end{align*}
As consequence of Kolmogorov continuity theorem (cf. Lemma~\ref{lemma01}), we have
$$\mathbb{E}\Bigg[\bigg\|\int_0^{\cdot}\mathcal{P}_H\mathcal G(\vu _{\varepsilon})dW\bigg\|_{C^{\alpha}([0,T];L_{\text{div}}^2(\mathbb{T}^3))}^a\Bigg]\le\,C$$
for all $\alpha\in(\frac{1}{a},\frac{1}{2})$, and $a\,\textgreater\,2$.
Combining the previous estimates and using the embeddings $W^{1,2}(0,T)\subset C^{1/2}[0,T]$, and $L_{\text{div}}^2(\mathbb{T}^3)\subset W_{\text{div}}^{-3,2}(\mathbb{T}^3)$, we conclude
$$\mathbb{E}\bigg[\|\vu _{\varepsilon}\|_{C^\alpha([0,T];W_{\text{div}}^{-3,2}(\mathbb{T}^3))}\bigg]\le\,C(T),$$
for some $\alpha\,\textless\,\frac{1}{2}$.
Next, we recall the following compact embedding \cite[Chapter 1]{m}
$$C^\alpha([0,T];W_{\text{div}}^{-3,2})\cap L^\infty(0,T;L_{\text{div}}^2(\mathbb{T}^3))\subset\subset C_w([0,T];L_{\text{div}}^2(\mathbb{T}^3)),$$
to conclude that $\mu_{\vu _{\varepsilon}}$ is tight.
\end{proof}

\begin{Proposition}\label{rhoutight14}
The set $\{\mu_{C_{\varepsilon}}, \mu_{D_{\varepsilon}}, \mu_{E_{\varepsilon}},\mu_{F_\varepsilon};\,\varepsilon\in(0,1),\, k\ge 1\}$ is tight on $ \mathcal{Y}_{C} \times \mathcal{Y}_{D}  \times \mathcal{Y}_{E}\times\mathcal{Y}_{F}$.
\end{Proposition}

\begin{proof}
By making use of the \emph{a priori} bound \eqref{apv}, and the fact that all bounded sets in $L_{w*}^{\infty}(0,T; \mathcal{M}_b(\T^3))$ are relatively compact with respect to the weak-$*$ topology, we obtain the desired result.
\end{proof}

\begin{Proposition}\label{rhoutight1}
The set $\{\mu_{{\mathcal{V}}_{\varepsilon}};\,\varepsilon\in(0,1)\}$ is tight on $\mathcal{Y}_{\mathcal{V}}$.
\end{Proposition}

\begin{proof}
This follows from the compactness criterion in $\big(L_{w*}^{\infty}((0,T)\times \T^3; \mathcal{P}(\R^3)), w^* \big)$. To see that, define the set
\begin{align*}
B_M:= \Big\lbrace {\mathcal{V}} \in \big(L^{\infty}((0,T)\times \T^3; \mathcal{P}(\R^3)), w^* \big); 
\int_0^T \int_{\T^3} \int_{\R^3}|\xi_1|^{2} \,d{\mathcal{V}}_{t,x}(\xi)\,\D x\,\D t \le M    \Big\rbrace,
\end{align*}
which is relatively compact in $\big(L^{\infty}((0,T)\times \T^3; \mathcal{P}(\R^3)), w^* \big)$. Notice that
\begin{align*}
\mathcal{L}[{\mathcal{V}}_{\varepsilon}](B^c_M)&=
\p\Bigg(\int_0^T \int_{\T^3} \int_{\R^3} \Big(|\xi_1|^{2}  \,d{\mathcal{V}}_{t,x}(\xi)\,\D x\,\D t > M  \Bigg) \\
&= \p\Bigg(\int_0^T \int_{\T^3} |\vu _{\varepsilon}|^2\,\D x\,\D t >M \Bigg)
 \le \frac1M \E\Big[\| \vu _{\varepsilon}\|_{L^2(\T^3)}^2\Big] \le \frac CM.
\end{align*}
The finishes the proof.
\end{proof}

\begin{Proposition}\label{rhoutight12}
The set $\{\mu_{X_{\varepsilon}};\,\varepsilon\in(0,1)\}$ is tight on $\mathcal{Y}_{X}$.
\end{Proposition}

\begin{proof}
To prove the result, it is enough to observe that the random variable $X_{\varepsilon}= \int_0^t\,\mathcal{P}_H\mathcal{G}(\bu_\varepsilon) \,\dif W(s) \in L^p\big(\Omega; W^{\alpha, q}(0,T; L_{\dv}^{2}(\T^3))\big)$, for $q\,\ge\,2$ (see \cite{FlandoliGatarek}). Therefore, a simple application of the compact embedding result given in Lemma \ref{comp} yields required tightness.
\end{proof}

\begin{Proposition}\label{rhoutight13}
The set $\{\mu_{Y_{\varepsilon}};\,\varepsilon\in(0,1)\}$ is tight on $\mathcal{Y}_{Y}$.
\end{Proposition}

\begin{proof}
It is easy to see that, $Y_\varepsilon(t)=\sum_{k\geq1}\int_0^t\intTor{ \vu_{\varepsilon} \cdot \mathcal{P}_H{\vc{G}_k}( \vu _\varepsilon ) } \,\Dif W$ is a square integrable martingale, for every $\varepsilon \in (0,1)$. Notice that for $a > 2$
\begin{align*}
\E\Big[ \Big|\sum_{k\ge 1} \int_s^t \int_{\T^3}\vu_{\varepsilon} \cdot\mathcal{P}_H {\vc{G}_k}(\vu _\varepsilon )\Big|^a \Big] &\le \E\Big[ \int_s^t \sum_{k=1}^{\infty} \Big|\int_{\T^3}\vu_{\varepsilon} \cdot {\vc{G}_k}(\vu_\varepsilon)\Big|^2  \Big]^{a/2} \\
& \le |t-s|^{a/2}\, \Big(1 + \E \Big[\sup_{0\le t \le T} \|\vu_{\varepsilon}\|^a_{L^2(\T^3)} \Big]\Big)
\le C|t-s|^{a/2}.
\end{align*}
Therefore, we can apply the classical Kolmogorov continuity theorem (cf. Lemma~\ref{lemma01}) to conclude that, for some $\beta>0$
$$
\sum_{k\geq1}\int_0^t\intTor{ \vu_{\varepsilon} \cdot {\mathcal{P}_H\vc{G}_k}(\vu _\varepsilon ) } \,\Dif W \in L^a(\Omega; C^{\beta}(0,T; \R)).
$$
Therefore, using the well-known compact embedding of $C^{\beta}$ into $C^0$, tightness of law follows.
\end{proof}

Making use of results obtained from Proposition~\ref{prop:rhotight}, Proposition~\ref{rhoutight14}, Proposition~\ref{rhoutight1}, Proposition~\ref{rhoutight12}, and Proposition~\ref{rhoutight13}, we conclude that
\begin{Corollary}
The set $\{\mu^\varepsilon;\,\varepsilon\in(0,1)\}$ is tight on $\mathcal{Y}$. 
\end{Corollary}

Having secured all necessary tightness results, we can now apply Jakubowski-Skorokhod representation theorem (see also Motyl  \cite{BrzezniakHausenblasRazafimandimby}) to extract almost sure convergence on a new probability space. To that context, we infer the following result:

\begin{Proposition}\label{prop:skorokhod1}
There exists a subsequence $\mu^\varepsilon$ (not relabelled), a probability space $(\tilde\Omega,\tilde\mf,\tilde\prst)$ with $\mathcal{Y}$-valued Borel measurable random variables $(\tilde{\vu }_\varepsilon, \tilde W_\varepsilon, \tilde C_{\varepsilon}, \tilde D_{\varepsilon}, \tilde E_{\varepsilon}, \tilde X_{\varepsilon}, \tilde Y_{\varepsilon}, \tilde{F}_{\varepsilon}, \tilde \nu_{\varepsilon})$, $\varepsilon \in (0,1)$, and\\  $(\tilde{\vu },\tilde W, \tilde C, \tilde D, \tilde E, \tilde X, \tilde Y, \tilde{F}, \tilde \nu)$ such that 
\begin{enumerate}
 \item [(1)]the law of $( \tilde{\vu }_\varepsilon, \tilde W_\varepsilon, \tilde C_{\varepsilon}, \tilde D_{\varepsilon}, \tilde E_{\varepsilon}, \tilde X_{\varepsilon}, \tilde Y_{\varepsilon}, \tilde{F}_{\varepsilon}, \tilde \nu_{\varepsilon})$ is given by $\mu^\varepsilon$, $\varepsilon\in(0,1)$,
\item [(2)]the law of $(\tilde{\vu },\tilde W, \tilde C, \tilde D, \tilde E, \tilde X, \tilde Y, \tilde{F}, \tilde \nu)$, denoted by $\mu$, is a Radon measure,
\item [(3)]$( \tilde{\vu }_\varepsilon, \tilde W_\varepsilon, \tilde C_{\varepsilon}, \tilde D_{\varepsilon}, \tilde E_{\varepsilon}, \tilde X_{\varepsilon}, \tilde Y_{\varepsilon},\tilde{F}_{\varepsilon},\, \tilde \nu_{\varepsilon})$ converges $\,\tilde{\prst}$-almost surely to \\$(\tilde{\vu }, \tilde W, \tilde C, \tilde D, \tilde E, \tilde X, \tilde Y, \tilde{F}, \tilde \nu)$ in the topology of $\mathcal{Y}$, i.e.,
\begin{align*}
&\tilde{\vu }_\varepsilon\rightarrow \bar{\vu} \,\, \text{in}\, \,C_w([0,T]; L_{\dv}^2(\T^3)),\quad&\tilde W_\varepsilon \rightarrow \tilde W \,\, \text{in}\, \,C([0,T]; \mathcal{U}_0)),\\
&\tilde C_\varepsilon \rightarrow \tilde C \,\, \text{weak-$*$ in}\, \, L_{w*}^{\infty}(0,T; \mathcal{M}_b(\T^3)), \qquad & \tilde D_\varepsilon \rightarrow \tilde D \,\, \text{weak-$*$ in}\, \, L_{w*}^{\infty}(0,T; \mathcal{M}_b(\T^3)), \\
&\tilde Y_\varepsilon \rightarrow \tilde Y \,\, \text{in}\, \, C([0,T]; \R),\qquad
& \tilde E_\varepsilon \rightarrow \tilde E \,\, \text{weak-$*$ in}\, \, L_{w*}^{\infty}(0,T; \mathcal{M}_b(\T^3)), \\
& \tilde \nu_\varepsilon \rightarrow \tilde \nu \,\, \text{weak-$*$ in}\, \, L_{w*}^{\infty}((0,T)\times \T^3; \mathcal{P}(\R^3)),\qquad& \tilde X_\varepsilon \rightarrow \tilde X \,\, \text{in}\, \,C([0,T];W_{\dv}^{-1,2}(\mathbb{T}^3)), \\
&\tilde{F}_\varepsilon\rightarrow\tilde{F}\,\,\text{weak-$*$ in}\,\,L_{w*}^{\infty}(0,T; \mathcal{M}_b(\T^3)),
\end{align*}
\item [(4)] For any $\varepsilon$, $\tilde W_{\varepsilon}=\tilde W$, $\tilde{\p}$-a.s.
\item [(5)] For any Carath\'{e}odory function $J=J(t,x,\vu)$, where $(t,x)\in (0,T)\times \T^3$ and $\vu \in \R^3$, satisfying for some $p$ the growth condition
$|J(t,x,\vu)| \le 1 + |\vu|^{p}$, uniformly in $(t,x)$. Then we have $\tilde\p$-a.s.
$$
J(\tilde\vu_\varepsilon) \rightarrow \overline{J(\tilde\vu)}\,\, \text{in}\,\, L^r((0,T)\times\T^3),\,\, \text{for all}\,\, 1<r\le\frac{2}{p}.
$$
\end{enumerate}
\end{Proposition}

\begin{proof}
Proof of the items $(1)$, $(2)$, and $(3)$ directly follow from Jakubowski-Skorokhod representation theorem. For the proof of the item $(4)$, we refer to Theorem \ref{newth}, and \cite{BrzezniakHausenblasRazafimandimby}. For the proof of the item $(5)$, we refer to the Lemma~\ref{lem001}.
\end{proof}

\subsubsection{Passing to the limit}

Note that, in view of the equality of joint laws, the energy inequality \eqref{eq:apriorivarepsilon} and the a priori estimate \eqref{apv} for the new random variables hold on the new probability space.
Making use of convergence results given by Proposition~\ref{prop:skorokhod1}, we can now pass to the limit in approximate equation \eqref{P1NS}, and the energy inequality \eqref{eq:apriorivarepsilon}. First we show that the approximations $\tilde \vu_{\varepsilon}$ solve the equation given by \eqref{P1NS} on the new probability space $(\tilde\Omega,\tilde\mf,\tilde\prst)$.
For that purpose, let us denote by $(\tilde{\mf}_t^\varepsilon)$ and $(\tilde{\mf}_t)$, $\tilde{\prst}$-augmented canonical filtrations of the process $(\tilde{\bf u}_\varepsilon,\tilde{W}_\varepsilon)$ and $(\langle {\mathcal{\tilde V}^{\omega}_{t,x}}; \tilde {\textbf u} \rangle,\tilde{W}, \tilde{X}, \tilde{Y})$, respectively. This means
\begin{equation*}
\begin{split}
\tilde{\mf}_t^\varepsilon&=\sigma\big(\sigma\big(\bfr_t\tilde{\bf u}_\varepsilon,\,\bfr_t \tilde{W}_\varepsilon\big)\cup\big\{N\in\tilde{\mf};\;\tilde{\prst}(N)=0\big\}\big),\quad t\in[0,T],\\
\tilde{\mf}_t&=\sigma\big(\sigma\big(\bfr_t\langle {\mathcal{\tilde V}^{\omega}_{t,x}}; \tilde {\textbf u} \rangle,\,\bfr_t\tilde{W},\, \bfr_t\tilde{X}, \bfr_t\tilde{Y}\big)\cup\big\{N\in\tilde{\mf};\;\tilde{\prst}(N)=0\big\}\big),\quad t\in[0,T],
\end{split}
\end{equation*}
where $\bfr_t$ is the restriction operator to the interval $[0,t]$ acting on various path spaces.

\begin{Proposition}\label{prop:martsol}
For every $\varepsilon\in(0,1)$, $\big((\tilde{\Omega},\tilde{\mf},(\tilde{\mf}_{\varepsilon,t})_{t\ge0},\tilde{\prst}),\tilde{\vu}_\varepsilon,\tilde{W}\big)$ is a finite energy weak martingale solution to \eqref{P1NS} with the initial law $\Lambda_\varepsilon$. 
\end{Proposition}

\begin{proof}
Proof of the above proposition is standard, and one can furnish the proof following the same line of argument, as in the monograph by Breit et. al. \cite[Theorem 2.9.1]{BrFeHobook}. For brevity, we skip all the details.
\end{proof}

\noindent We remark that, in light of the above proposition, the new random variables satisfy the following equations and the energy inequality on the new probability space
\begin{itemize}
\item for all  $\bm{\varphi}\in C_{\text{div}}^\infty(\mathbb{T}^3)$ we have
\begin{equation}\label{eq:energyt}
\begin{aligned}
\langle \mathbf{\tilde{u}}_{\varepsilon}(t), \bm\varphi\rangle &= \langle \mathbf{\tilde{u}}_{\varepsilon} (0), \bm{\varphi}\rangle - \int_0^{t}\langle \mathbf{\tilde{u}}_{\varepsilon}\otimes\mathbf{\tilde{u}}_{\varepsilon}, \nabla_x  \bm{\varphi}\rangle\,\mathrm{d}s
+ {\varepsilon}\,\int_0^{t} \langle\nabla_x \mathbf{\tilde{u}}_{\varepsilon}\,, \nabla_x  \bm{\varphi}\rangle\,\mathrm{d}s 
+\int_0^{t}\langle\mathcal{G}(\mathbf{\tilde{u}}_{\varepsilon}), \bm{\varphi}\rangle\,\mathrm{d}W
\end{aligned}
\end{equation}
$\mathbb{P}$-a.s. for all $t\in[0,T]$,
\item the energy inequality\index{energy inequality} 
\begin{align}\label{wq}
& -\int_0^T \partial_t\phi \int_{\T^3}  \frac{1}{2}|\tilde{ {\bf u}}_{\varepsilon} |^2\, \D x  \, \D t
+ {\varepsilon}\, \int_0^T \phi \int_{\T^3}|\nabla_x  \tilde{{\bf u}}_{\varepsilon}|^2 \, \D x\,  {\rm d}t \\
&\leq \phi(0) \int_{\T^n} \frac{1}{2}{|\bfu_{\varepsilon}(0)|^2}\, \D x
+\sum_{k=1}^\infty\int_0^T\phi\bigg(\int_{\T^n}\mathbf{G}_k (\tilde{{\bf u}}_{\varepsilon})\cdot\tilde{{\bf u}}_{\varepsilon}\dx\bigg){\rm d} W_k + \frac{1}{2}\sum_{k = 1}^{\infty}  \int_0^T\phi
\int_{\T^n}|\mathcal{P}_H \mathbf{G}_k (\tilde{{\bf u}}_{\varepsilon}) |^2 \, {\rm d}t \notag
\end{align}
holds $\mathbb P$-a.s., for all $\phi\in C_c^\infty([0,T)),\,\phi\,\ge\,0$.
\end{itemize}

\noindent Now we are in a position to pass to the limit in $\varepsilon$ in \eqref{eq:energyt} and \eqref{wq}. To see this, note that we have a priori estimate \eqref{apv} for the new random variable. Therefore, an application of Lemma~\ref{lem001} helps us to conclude that $\tilde \p$-a.s., 
\begin{align*}
\tilde {\textbf u}_\varepsilon \rightharpoonup \langle {\mathcal{\tilde V}^{\omega}_{t,x}}; \tilde {\textbf u} \rangle, \,\,\text{weakly in}\,\, L^2((0,T);L_{\text{div}}^2(\T^3)).
\end{align*}
Moreover, making use of item $(5)$ of Lemma~\ref{prop:skorokhod1}, we conclude that $\mathbb{P}$-a.s.
$$
\mathbf{G}_k (\tilde {\bf u}_\varepsilon) \rightharpoonup \big \langle \tilde {\mathcal{V}}_{t,x}^\omega ; \mathbf{G}_k (\tilde {\bf u}) \big \rangle \,\, \mbox{weakly in} \,\,L^2((0,T); L^{2}(\T^3)).
$$
This, in particular, implies that $\p$-a.s.
$$
\mathcal{Q}_H\mathbf{G}_k (\tilde {\bf u}_\varepsilon) \rightharpoonup  \mathcal{Q}_H\big \langle \tilde{\mathcal{V}}_{t,x}^\omega ; \mathbf{G}_k (\tilde {\bf u}) \big \rangle \,\, \mbox{weakly in} \,\,L^2((0,T); (L^{2}(\T^3))^\perp).
$$
Indeed, Let $\mathbf{v}\in L^2((0,T);(L^2(\T^3))^\perp)$ then orthogonal property of projection $\mathcal{Q}_H$ implies that
$$\lim_{\varepsilon\to 0}\langle \mathcal{Q}_H\mathbf{G}_k (\tilde {\bf u}_\varepsilon), \mathbf{v}\rangle=\lim_{\varepsilon\to 0}\langle \mathbf{G}_k (\tilde {\bf u}_\varepsilon), \mathbf{v}\rangle=\langle\langle \tilde{\mathcal{V}}_{t,x}^\omega ; \mathbf{G}_k (\tilde {\bf u}) \big \rangle, \mathbf{v}\rangle=\langle\mathcal{Q}_H\langle \tilde{\mathcal{V}}_{t,x}^\omega ; \mathbf{G}_k (\tilde {\bf u}) \big \rangle, \mathbf{v}\rangle $$

As usual, to identify the weak limits related to the nonlinear terms present in the equations, we first need to introduce corresponding concentration defect measures 
\begin{align*}
\tilde \mu_{C}& = \tilde C -\left\langle \mathcal{\tilde V}^{\omega}_{(\cdot, \cdot)}; {\tilde {\bf u}\otimes \tilde {\bf u}} \right\rangle \, \D x \, \D t,\,\,\tilde \mu_{E} = \tilde E- \left\langle \mathcal{\tilde V}^{\omega}_{(\cdot, \cdot)}; \frac{1}{2} {|\tilde {\bf u}|^2} \right \rangle \,\D x\, \D t, \\
\tilde \mu_{D}& = \tilde D -\left\langle \mathcal{\tilde V}^{\omega}_{(\cdot, \cdot)}; \sum_{k \geq 1} |{\vc G}_k (\tilde {\bf u}) |^2 \right\rangle \, \D x \, \D t, \,\,
\tilde \mu_{F} = \tilde F -\Big|\mathcal{Q}_H \left\langle \mathcal{\tilde V}^{\omega}_{(\cdot, \cdot)}; \sum_{k \geq 1} |{\vc G}_k (\tilde {\bf u}) |^2 \right\rangle\Big|^2 \, \D x\, \D t.
\end{align*}
In view of the discussion in Subsection~\ref{ym}, and making use of above concentration defect measures, we can conclude that $ \mathbb{\tilde P}$ almost surely 
\begin{align*}
&\tilde C_\varepsilon \rightharpoonup \left\langle \mathcal{\tilde V}^{\omega}_{(\cdot, \cdot)}; {\tilde {\bf u}\otimes \tilde {\bf u}}\right\rangle \, \D x\, \D t + \tilde \mu_{C}, \,\, \text{ weak-$*$ in}\, \, L_{w*}^{\infty}(0,T; \mathcal{M}_b(\T^3)), \\
&\tilde D_\varepsilon \rightharpoonup \left\langle \mathcal{\tilde V}^{\omega}_{(\cdot, \cdot)}; \sum_{k \geq 1}|{\vc G}_k (\tilde {\bf u}) |^2\right\rangle \, \D x \, \D t + \tilde \mu_{D}, \,\, \text{weak-$*$ in}\, \, L_{w*}^{\infty}(0,T; \mathcal{M}_b(\T^3)), \\
&\tilde E_\varepsilon \rightharpoonup \left\langle \mathcal{\tilde V}^{\omega}_{(\cdot, \cdot)}; \frac{1}{2} {|\tilde {\bf u}|^2}\right\rangle \, \D x \, \D t + \tilde \mu_{E}, \,\, \text{weak-$*$ in}\, \, L_{w*}^{\infty}(0,T; \mathcal{M}^+_b(\T^3)),\\
&\tilde F_\varepsilon \rightharpoonup \Big|\mathcal{Q}_H \left\langle \mathcal{\tilde V}^{\omega}_{(\cdot, \cdot)}; \sum_{k \geq 1} |{\vc G}_k (\tilde {\bf u}) |^2 \right\rangle\Big|^2 \, \D x\, \D t +\tilde{\mu}_F, \,\, \text{weak-$*$ in}\, \, L_{w*}^{\infty}(0,T; \mathcal{M}^+_b(\T^3)).
\end{align*}

\noindent  Note that both defect measures $\tilde{\mu}_E$, and $\tilde{\mu}_F$ are positive, thanks to the lower semi-continuity property of norms. Next, we move on to the martingale terms $\tilde X_{\varepsilon}$, appearing in the momentum equation, and $\tilde Y_{\varepsilon}$, appearing in the energy inequality. Regarding convergence of these terms, we state following propositions.
\begin{Proposition}
For every time $t \in [0,T]$, $\p$-almost surely $\tilde Y_{\varepsilon}(t) \rightarrow \tilde Y(t)$ in $\R$, where $\tilde Y(t)$ is a real valued square-integrable martingale with respect to the filtration $(\tilde{\mf}_t)$.
\end{Proposition}

\begin{proof}
First of all, in view of the Proposition~\ref{prop:skorokhod1}, we conclude that $\tilde Y_{\varepsilon} \rightarrow  \tilde Y $ $\p$-a.s. in $C([0,T]; \R)$. To claim that $\tilde Y(t)$ is a martingale, as usual, it is sufficient to show that
$$
\tilde \E[\tilde Y(t)| \mathcal{\tilde F}_s] = \tilde Y(s),
$$
for all $t,s \in [0,T]$ with $s \le t$. In other words, it is enough to prove that
$$
\tilde \E \Big[ \mathfrak{L}_s(\tilde \Phi) \big(\tilde Y(t)-\tilde Y(s)\big) \Big]=0,
$$
where we denote $\tilde \Phi := (\langle {\mathcal{\tilde V}^{\omega}_{t,x}}; \tilde {\textbf u} \rangle,\tilde{W}, \tilde{X}, \tilde{Y})$, and on the path space $\underline{\mathcal{Y}} :=\mathcal{Y}_{{\bf u}} \times \mathcal{Y}_{W} \times \mathcal{Y}_{X} \times \mathcal{Y}_{Y}$, we denote $\mathfrak{L}_s$ by any bounded continuous functional which depends on on the values of $\tilde \Phi$ restricted to $[0,s]$. The idea is to use the fact that $\tilde Y_{\varepsilon}(t)$ is a martingale, i.e.,
$$
\tilde \E \Big[ \mathfrak{L}_s(\tilde \Phi_{\varepsilon}) \big(\tilde Y_{\varepsilon}(t)-\tilde Y_{\varepsilon}(s)\big) \Big]=0,
$$
for all bounded continuous functional $\mathfrak{L}_s$ on the same path space, and $\tilde \Phi_{\varepsilon} = (\tilde {\bf u}_{\varepsilon}, \tilde W, \tilde X_\varepsilon, \tilde Y_\varepsilon)$. At this point, we recall  Proposition~\ref{prop:skorokhod1} to conclude that $\tilde \Phi_{\varepsilon} \rightarrow \tilde \Phi$, $\p$-a.s. in the (weak) topology of $\underline{\mathcal{Y}}$. This, in particular, implies that $\mathfrak{L}_s(\tilde \Phi_{\varepsilon}) \rightarrow \mathfrak{L}_s(\tilde \Phi)$ $\p$-a.s. 
Now given this information, along with the fact that $\tilde Y_{\varepsilon}(t) \in L^2(\tilde \Omega)$, we may apply classical Vitali's convergence theorem to pass to the limit in $\varepsilon$ to conclude that $\tilde Y(t)$ is a martingale. 
\end{proof}
\begin{Proposition}For each time $t\in [0,T]$, $\tilde X_{\varepsilon}(t) \rightarrow \tilde X(t)$, $\p$-almost surely in the topology of $H_{\dv}^{-1}(\T^3)$. Moreover, $\tilde {X}(t)$ is also a $H_{\dv}^{-1}(\T^3)$-valued square integrable martingale with respect to the filtration $(\tilde{\mf}_t)$.
\end{Proposition}

\begin{proof}
Note that, as before, we would not be able to identify the structure of the martingale $\tilde X(t)$, instead we just prove that $\tilde X(t)$ is a martingale. In what follows, with the help of the Proposition~\ref{prop:skorokhod1}, we can conclude that for each $t \in [0,T]$, $ \tilde X_{\varepsilon}(t) \rightarrow \tilde X(t)$, $\p$-a.s. in the topology of $H^{-1}_{\text{div}}(\T^3)$. Fianlly to show that $\tilde X(t)$ is a martingale, it is enough to demonstrate that for all $i\ge1$
$$
\tilde \E \Big[ \mathfrak{L}_s(\tilde \Phi) \big <\tilde X(t)-\tilde X(s), g_i \big> \Big]=0,
$$
where $g_i$'s are given orthonormal basis for the space $H_{\dv}^{-1}(\T^3)$. We follow the usual argument to establish the result. To that ocntext, we first use the information that 
$$
\tilde \E \Big[ \mathfrak{L}_s(\tilde \Phi_{\varepsilon}) \big<\tilde X_{\varepsilon}(t)-\tilde X_{\varepsilon}(s), g_i \big> \Big]=0,
$$
for all $i \ge 1$. Then, like before, we can pass to the limit in the parameter $\varepsilon$ to show that $\tilde M(t)$ is a martingale. Indeed, this argument requires uniform integrabilty in $\omega$ variable, and can be achieved using BDG inequality:
\begin{align*}
\tilde \E \Big[\big|\big< \tilde X_\varepsilon(t), g_i \big>\big|^p  \Big] &=\tilde \E \bigg[\bigg|\Big< \int_0^t\mathcal{P}_H{\mathcal{G}}( \tilde {\textbf u}_\varepsilon )\,\Dif \tilde W, g_i \Big>\bigg|^p  \bigg] 
 \le C \, \tilde \E \Bigg[\sup_{0\le t \le T} \bigg \| \int_0^t \mathcal{P}_H {\mathcal{G}}( \tilde {\textbf u}_\varepsilon )\,\Dif \tilde W \bigg\|^p_{H_{\dv}^{-1}(\T^3)}\Bigg]  \\
&\le C \,\tilde \E \Bigg[\bigg(\int_0^T \| {\mathcal{G}}(\tilde {\textbf u}_\varepsilon )\|^2_{L_2(\mathcal{U}, L^{2}(\T^3))} \,\D s \bigg)^{p/2}\Bigg] \le C.
\end{align*}
This finishes the proof.
\end{proof}
\noindent In view of the above discussions, we can pass to the limit in \eqref{eq:energyt} to conclude that
\begin{equation*} 
\begin{aligned}
&\int_{\T^3} \langle \tilde{\mathcal{V}}^{\omega}_{\tau,x};\vu \rangle \cdot \bm{\varphi}(\tau, \cdot) \,\D x - \int_{\T^3} \langle \tilde{\mathcal{V}}^{\omega}_{0,x};\vu \rangle \cdot \bm{\varphi}(0,\cdot) \, \D x \\
&\qquad = \int_{0}^{\tau} \int_{\T^3}\langle \tilde{\mathcal{V}}^{\omega}_{t,x}; {\vu \otimes \vu }\rangle: \nabla_x \bm{\varphi} \, \D x \, \D t + \int_{\T^3} \bm{\varphi}\,\int_0^{\tau} d\tilde{X}(t) \,\D x+ \int_{0}^{\tau} \int_{\T^3}  \nabla_x \bm{\varphi}: d\tilde{\mu}_C,
\end{aligned}
\end{equation*}
holds  $\tilde{\p}$-a.s., for all $\tau \in [0,T)$, and for all $\bm{\varphi} \in C_{\text{div}}^{\infty}(\T^3;\mathbb{R}^3)$. This implies that \eqref{second condition measure-valued solution} holds. Next, we focus on proving the energy inequality \eqref{third condition measure-valued solution}. To that context, making use of identifications of weak limits of various terms involved in the energy inequality, we can pass to the limit in $\varepsilon$ in \eqref{wq}. This yield
\begin{align}\label{tr}
&-\int_{0}^{T}\partial_t\phi\bigg(\int_{\mathbb{T}^3}\left\langle \mathcal{\tilde{V}}^{\omega}_{\tau,x};\frac{|{\bf u}|^2}{2} \right\rangle \dx +\mathcal{\tilde{D}}(\tau)\bigg){\rm d}\tau 
\le\,\phi(0)\int_{\mathbb{T}^3}\left\langle \mathcal{\tilde{V}}^{\omega}_{0,x};\frac{|{\bf u}|^2}{2} \right\rangle \, \D x \notag \\
&\quad + \frac{1}{2} \sum_{k\ge\,1} \int_0^{T}\phi(\tau) \int_{\T^3} \left\langle \mathcal{V}^{\omega}_{\tau,x};{|{\textbf G_k( u)}|^2} \right\rangle \,\D x\, \D \tau 
-\frac{1}{2}\sum_{k\ge\,1}\int_s^t \phi(\tau) \int_{\mathbb{T}^3}\Big(\mathcal{Q}_H\left\langle \tilde{\mathcal{V}}^{\omega}_{\tau,x};{|{\textbf G_k( u)}|}\right\rangle\Big)^2\, \D x\, {\rm d}\tau \notag \\
& \qquad +\frac{1}{2}\int_{0}^{T}\phi(\tau)\int_{\mathbb{T}^3}{\rm d}\tilde\mu_D \, \D \tau
- \frac{1}{2}\int_{0}^{T}\phi(\tau)\int_{\mathbb{T}^3}{\rm d}\tilde\mu_F\, \D \tau
+\int_{0}^{T}\phi(\tau)\,{\rm d}\tilde{Y}(\tau)
\end{align}
holds $\tilde{\mathbb P}$-a.s., for all $\phi\in C_c^\infty([0,T)),\,\phi\,\ge\,0$. Note that here $\mathcal{\tilde D}(\tau):= \tilde \mu_{E}(\tau)(\T^3)$. To proceed, we first fix any $s$ and $t$ such that
$0\,\textless\,s\,\textless\,t\,\textless\,T$. For any $r\,\textgreater\,0$ with $0\,\textless\,s-r\textless\,t+r\,\textless\,T$, let us denote by $\phi_r$, a Lipschitz fucntion which is linear on $[s-r,s]$ and $[t, t+r]$ such that
$$\phi_r(\tau)=\begin{cases}
	0,&\text{if}\,\, \tau\in[0,s-r]\,\,\text{or}\,\,\tau\in[t+r,T]\\
	1,&\text{if}\,\,\tau\in[s,t].
\end{cases}$$
Then, a standard regularization argument reveals that $\phi_r$ can be used as an admissible test fuction in \eqref{wq}. Therefore, replacing the test function $\phi$ by $\phi_r$ in \eqref{tr}, we get $\p$-a.s., for all $0\,\textless\,s\,\textless\,t\,\textless\,T$
\begin{align}\label{pr}
&\frac{1}{r}\int_{t}^{t+r}\bigg(\int_{\mathbb{T}^3}\left\langle \mathcal{\tilde{V}}^{\omega}_{\tau,x};\frac{|{\bf u}|^2}{2} \right\rangle \dx +\mathcal{\tilde{D}}(\tau)\bigg){\rm d}\tau
\le\,\frac{1}{r}\int_{s-r}^s\bigg(\int_{\mathbb{T}^3}\left\langle \mathcal{\tilde{V}}^{\omega}_{\tau,x};\frac{|{\bf u}|^2}{2} \right\rangle \dx+\mathcal{\tilde{D}}(\tau)\bigg){\rm d}\tau \notag \\
& \quad + \frac{1}{2} \sum_{k\ge\,1} \int_{s-r}^{t+r}\phi_r(\tau) \int_{\T^3} \left\langle \mathcal{V}^{\omega}_{\tau,x};{|{\textbf G_k( u)}|^2} \right\rangle \, \D x \, \D \tau -\frac{1}{2} \sum_{k\ge\,1} \int_{s-r}^{t+r}\phi_r(\tau) \int_{\mathbb{T}^3}\Big(\mathcal{Q}_H\left\langle \tilde{\mathcal{V}}^{\omega}_{\tau,x};{|{\textbf G_k( u)}|}\right\rangle\Big)^2\, \D x\,{\rm d}\tau\notag\\
&\qquad +\frac{1}{2}\int_{s-r}^{t+r}\phi_r(\tau) \int_{\mathbb{T}^3} {\rm d}\tilde\mu_D \, \D \tau  -\frac{1}{2}\int_{s-r}^{t+r} \phi_r(\tau) \int_{\mathbb{T}^3} \D \tilde{\mu}_F\, {\rm d}\tau +\int_{s-r}^{t+r}\phi_r(\tau)\,{\rm d}\tilde Y(\tau).
\end{align}
Now using the non-negativity of the defect measure $\tilde{\mu}_F$, and letting $r\to 0^+$ in \eqref{pr}, we obtain $\p$-a.s., for all $0\,\textless\,s\,\textless\,t\,\textless\,T$
\begin{align}\label{ps}
&\lim_{r\to 0^+}\frac{1}{r}\int_{t}^{t+r}\bigg(\int_{\mathbb{T}^3}\left\langle \mathcal{\tilde{V}}^{\omega}_{\tau,x};\frac{|{\bf u}|^2}{2} \right\rangle \dx +\mathcal{\tilde{D}}(\tau)\bigg){\rm d}\tau\,\notag\\&\qquad\le\,\lim_{r\to0 ^+}\frac{1}{r}\int_{s-r}^s\bigg(\int_{\mathbb{T}^3}\left\langle \mathcal{\tilde{V}}^{\omega}_{\tau,x};\frac{|{\bf u}|^2}{2} \right\rangle \dx+\mathcal{\tilde{D}}(\tau)\bigg){\rm d}\tau+\frac{1}{2} \sum_{k\ge\,1} \int_{s}^{t}\int_{\T^3} \left\langle \mathcal{V}^{\omega}_{\tau,x};{|{\textbf G_k( u)}|^2} \right\rangle \dx\,{\rm d}\tau\notag\\&\qquad\qquad-\frac{1}{2} \sum_{k\ge\,1} \int_{s}^{t}\int_{\mathbb{T}^3}\Big(\mathcal{Q}_H\left\langle \tilde{\mathcal{V}}^{\omega}_{\tau,x};{|{\textbf G_k( u)}|}\right\rangle\Big)^2\dx\,{\rm d}\tau+\frac{1}{2}\int_{s}^{t}\int_{\mathbb{T}^3}{\rm d}\tilde\mu_D \, \D \tau +\int_{s}^{t}{\rm d}\tilde Y(\tau)
\end{align}

%
We remark that for $s=0$ we need to use a slightly different test function to conclude the result. In this case we take 
$$\phi_r(\tau)=\begin{cases}
	1,&\text{if}\,\,\tau\in[0,t]\\
	\text{linear},&\text{if}\,\,\tau\in[t,t+r]\\
	0,&\text{otherwise}.
\end{cases}$$
and apply the same argument as before to establish that the energy inequality \eqref{third condition measure-valued solution} holds.

Now we are only left with the verifications of \eqref{fourth condition measure-valued solutions}, and item (i) of Definition~\ref{def:dissMartin}. To proceed, we start with the following lemma.

\begin{Lemma}\label{rhoutight1311}
Given a stochastic process $h$, as in item (i) of Definition~\ref{def:dissMartin}
$$
\mathrm{d}h  = D^d_th\,\mathrm{d}t  + \mathbb{D}^s_th\,\mathrm{d} \tilde W,
$$
the cross variation with $\tilde X$ is given by 
\begin{align*}
	\Big<\hspace{-0.14cm}\Big<h(t), \tilde {X}(t)  \Big>\hspace{-0.14cm}\Big> 
	= \sum_{i,j}\Bigg(\sum_{k = 1}^{\infty}  \int_0^t\langle \mathcal{P}_H \left\langle \tilde{ \mathcal{V}}^{\omega}_{s,x}; \mathbf{G}_k (\vu )\right\rangle, g_i\rangle \,\langle\mathbb{D}_t^s h(e_k), g_j\rangle \,ds\Bigg) g_i\otimes g_j .
\end{align*}
where $g_i$'s are orthonormal basis for $H^{-1}_{\dv}(\mathbb{T}^3)$ and bracket $\langle\cdot,\cdot\rangle$ denotes inner product in the same space.\end{Lemma}

\begin{proof}
Following the definition of cross variation between two Hilbert space valued martingales, given by Da Prato $\&$ Zabczyk \cite[Section 3.4]{prato}, we have 
\begin{align*}
\Big<\hspace{-0.14cm}\Big<h(t),  \tilde{X}_\varepsilon(t)  \Big>\hspace{-0.14cm}\Big> =\sum_{i,j}\Big<\hspace{-0.14cm}\Big<\big\langle h(t),g_i\big\rangle, \big\langle \tilde{X}_\varepsilon(t),g_j\big\rangle  \Big>\hspace{-0.14cm}\Big> g_i\otimes g_j, 
\end{align*}
where using the informations of the processes $h(t)$ and $\tilde{X}_\varepsilon(t)$, we have
\begin{align*}
\Big<\hspace{-0.14cm}\Big<\big\langle h(t),g_i\big\rangle, \big\langle \tilde{X}_\varepsilon(t),g_j\big\rangle  \Big>\hspace{-0.14cm}\Big>&=\Big<\hspace{-0.14cm}\Big< \sum_{k\ge 1}\int_0^t\big \langle \mathbb{D}^s_t h(e_k),g_i\big \rangle \D \tilde{W}_k, \sum_{k\ge 1}\int_0^t\big\langle \mathcal{P}_H\textbf{G}_k(\tilde{\vu}_\varepsilon),g_j\big\rangle \D \tilde{W}_k  \Big>\hspace{-0.14cm}\Big> \\
&=\sum_{k\ge\,1}\int_0^t \big\langle \mathbb{D}^s_t h(e_k),g_i\big \rangle\,\big\langle \mathcal{P}_H\textbf{G}_k(\tilde{\vu}_\varepsilon),g_j\big\rangle \, \D s
\end{align*}
Therefore we get
\begin{align*}
\Big<\hspace{-0.14cm}\Big<h(t),  \tilde{X}_\varepsilon(t)  \Big>\hspace{-0.14cm}\Big> 
= \sum_{i,j}\Bigg(\sum_{k = 1}^{\infty}  \int_0^t \big<\mathbb{D}^s_th(e_k), h_i \big> \, \big<\mathcal{P}_H\mathbf{G}_k (\tilde {\bf u}_{\varepsilon}), h_j \big> \,\D s\Bigg) h_i\otimes h_j .
\end{align*}
This equivalently implies that
\begin{align*}
&\tilde \E \Big[ \mathfrak{L}_s(\tilde \Phi_{\varepsilon}) \Big( \big<h(t), g_i \big>\big<\tilde X_{\varepsilon}(t), g_j \big>- \sum_{k = 1}^{\infty} \int_0^t  \big<\mathbb{D}^s_th(e_k), g_i \big> \, \big<\mathcal{P}_H\mathbf{G}_k (\tilde {\bf u}_{\varepsilon}), g_j \big> \,\D s \Big)\Big]=0.
\end{align*}
Since we have $\mathbb{P}$-a.s.
$$
\mathbf{G}_k (\tilde {\bf u}_{\varepsilon}) \rightharpoonup
\big \langle \tilde{\mathcal{V}}_{t,x}^\omega ; \mathbf{G}_k (\tilde {\bfu }) \big \rangle, \,\, \mbox{weakly in} \,\,L^2((0,T);L^2(\T^3)),
$$ 
This gives $\mathbb{P}$-a.s
$$
\mathcal{P}_H\mathbf{G}_k (\tilde {\bf u}_{\varepsilon}) \rightharpoonup
\mathcal{P}_H\big \langle \tilde{\mathcal{V}}_{t,x}^\omega ; \mathbf{G}_k (\tilde {\bfu }) \big \rangle, \,\, \mbox{weakly in} \,\,L^2((0,T);L^2_{\dv}(\T^3)).
$$
Moreover, making use of the a priori estimate \eqref{apv}, we have
\begin{align*}
\tilde \E \Big[\int_0^T \|\mathcal{P}_H \mathcal{G}(\tilde {\bf u}_{\varepsilon}) \|^2_{L_2(\mathcal{U}; H_{\dv}^{-1}(\T^3))}\,dt \Big]\le \tilde \E \Big[\int_0^T \int_{\T^3} (\tilde 1+|\tilde {\bf u}_{\varepsilon}|^2)\,dx\,dt \Big]\le C.
\end{align*}
This implies that $\p$-a.s.
$$
\mathcal{P}_H\mathbf{G}_k (\tilde {\bf u}) \rightharpoonup  \mathcal{P}_H\big \langle \tilde {\mathcal{V}}_{t,x,}^{\omega} ; \mathbf{G}_k (\tilde {\bf u}) \big \rangle, \,\, \mbox{weakly in} \,\,L^2((0,T); H_{\dv}^{-1}(\T^3)).
$$ 
Therefore we can pass to the limit in $\ep \rightarrow 0$, thanks to uniform integrabilty, to conclude
\begin{align*}
\tilde \E \Big[ \mathfrak{L}_s(\tilde \Phi) & \Big( \big<h(t), g_i \big>\big<\tilde X(t), g_j \big>- \sum_{k = 1}^{\infty} \int_0^t  \big<\mathbb{D}^s_th(e_k), g_i \big> \, \big<\mathcal{P}_H\left\langle \mathcal{\tilde V}^{\omega}_{s,x}; \mathbf{G}_k (\tilde {\bf u})\right\rangle, g_j \big> \,\D s \Big)\Big]=0.
\end{align*}
This finishes the proof the lemma.
\end{proof}

Finally, regarding the verifications of \eqref{fourth condition measure-valued solutions}, we have the following lemma:
\begin{Lemma}\label{rhoutight131}
Concentration defect measures $\tilde \mu_{C}$, and $\tilde \mu_{D}$ are dominated by the nonnegative concentration defect measures $\mathcal{\tilde D}(\tau):= \tilde \mu_{E}(\tau)(\T^3)$, in the sense of Lemma~\ref{lemma001}. More precisely, there exists a constant $K>0$ such that
\begin{equation*} 
\int_{0}^{\tau} \int_{\T^3} d|\tilde \mu_C| + \int_{0}^{\tau} \int_{\T^3} d|\tilde \mu_D|\leq K \int_{0}^{\tau} \mathcal{\tilde D}(\tau)\,dt,
\end{equation*}	
$\p$-a.s., for all $\tau \in (0,T)$.
\end{Lemma}

\begin{proof}
Clearly, by Lemma~\ref{lemma001}, we can conclude that $\tilde \mu_{E}$ dominates the defect measure $\tilde \mu_{C} $. On the other hand, making use of hypotheses (\ref{FG1}), (\ref{FG2}), we can show the required dominance of $\tilde \mu_{E}$ over $\tilde \mu_{D}$. Indeed, note that the function
\[
\vc{u} \mapsto \sum_{k \geq 1}{ |{\vc G}_k ({\bf u}) |^2 } \ \mbox{is continuous},
\]
and clearly dominated by the total energy
\[
\sum_{k \geq 1} { |{\vc G}_k ({\bf u}) |^2 }\leq C \left( 1 + {| {\bf u} |^2} \right) 
\]
This finishes the proof of the lemma. 
\end{proof}


\section{Weak-Strong Uniqueness Principle for Euler System}
\label{proof2}
In this section we prove Theorem~\ref{Weak-Strong Uniqueness} through auxiliary results. Essentially the proof relies upon successful identification of the cross variation between two processes, given by a measure-valued Euler solution and a local strong Euler solution. In what follows, we begin with the following lemma.
\begin{Lemma}[Weak It\^o Product Formula]
\label{Ito}
Let $\mathbf{q}$ be a stochastic process on $\big(\Omega,\mathbb{F}, (\mathbb{F}_{t})_{t\geq0},\mathbb{P} \big)$ such that 
$$ \mathbf{q}\in C_w([0,T];L^2_{\dv }(\mathbb{T}^3))\cap L^{\infty}((0,T);L^2_{\dv}(\mathbb{T}^3)), \,\,\mathbb{P}-\text{a.s.} $$
$$\mathbb{E}\Big[\sup_{t\in[0,T]}\|\mathbf{q}\|_{L^2_{\dv}(\mathbb{T}^3)}^2\Big]\,\textless\,+\infty.$$
Moreover, it satisfies $\mathbb{P}$-a.s.
\begin{align}\label{eq1}
\int_{\mathbb{T}^3}\mathbf{q}(t)\cdot\bm{\varphi}\,dx=\int_{\mathbb{T}^3}\textbf{q}(0)\cdot\bm{\varphi}\,dx + \int_{0}^t\int_{\mathbb{T}^3}\mathbf{q}_1:\nabla\bm{\varphi}dx ds+\int_0^t\int_{\mathbb{T}^3}\nabla\bm{\varphi}:\,d\mu(x,s) \,ds + \int_{\mathbb{T}^3}\bm{\varphi}\cdot\int_0^tdM dx
\end{align}
for all $t\in[0,T]$, and test function $\varphi \in C^{\infty}_{\mathrm{div}}(\T^3)$. Here $M$ is a continuous square integrable $H^{-1}_{\dv}(\mathbb{T}^3)$-valued martingale and $\mathbf{q}_1,\,\,\mu $ are progressible measurable with
 $$\mathbf{q}_1\in L^2(\Omega;L^1(0,T;L^2_{\dv}(\mathbb{T}^3))),\qquad \mu\in L^1(\Omega;L_{w*}^\infty(0,T;\mathcal{M}_b(\mathbb{T}^3))).$$
Let $\textbf{Q}$ be a stochastic process on $\big(\Omega,\mathbb{F}, (\mathbb{F}_{t})_{t\geq0},\mathbb{P} \big)$ satisfying
$$\textbf{Q}\in C([0,T]; C^1(\mathbb{T}^3)),\,\mathbb{P}-\text{a.s.}\,\,\text{and}\,\,\,\,\mathbb{E}\big[\sup_{t\in[0,T]}\|\mathbf{Q}\|_{L_{\dv}^2(\mathbb{T}^3)\cap {C}(\mathbb{T}^3)}^2\big]\textless\,\infty,$$
be such that
$$\D\mathbf{Q}={\mathbf{Q}_1}{\rm d}t + \mathbf{Q}_2 \,\D W. $$
Here $\mathbf{Q}_1, \mathbf{Q}_2$ are progressible measurable with
\begin{align*}
\mathbf{Q}_1\in L^2(\Omega;L^1((0,T); & L^2_{\dv}(\mathbb{T}^3))),\,\qquad\,\textbf{Q}_2\in L^2(\Omega;L^2((0,T);L_2(\mathfrak{U};L_{\dv}^2(\mathbb{T}^3)))), \\
&\sum_{k=1}^\infty\int_0^T\|\mathbf{Q}_2(e_k)\|^2_{L^2(\mathbb{T}^3)}\in L^1(\Omega).
\end{align*}
Then $\mathbb{P}$-a.s., for all $t\in[0,T]$,
\begin{align}\label{eq4}
\int_{\mathbb{T}^3}\mathbf{q}(t)\cdot\mathbf{Q}(t)\dx&=\int_{\mathbb{T}^3}\mathbf{q}(0)\cdot\mathbf{Q}(0)\dx + \int_0^t\int_{\mathbb{T}^3} \mathbf{q}_1:\nabla \mathbf{Q} \dx\,{\rm d}s+\int_0^t\int_{\mathbb{T}^3}\nabla\mathbf{Q}:\,\D\mu\,\D s\notag\\
&\qquad+\int_{\mathbb{T}^3} \int_0^t\mathbf{Q}\cdot \D M\,\D x+\int_0^t\int_{\mathbb{T}^3}\mathbf{Q}_1\cdot\mathbf{q}\,\D x\, \D s + \int_{\mathbb{T}^3} \int_0^t\mathbf{q}\cdot\mathbf{Q}_2\,\D W\,\D x\notag\\
&\qquad+\int_{\mathbb{T}^3}\Big<\hspace{-0.14cm}\Big<M(t)\,,\, \mathbf{Q}(t)  \Big>\hspace{-0.14cm}\Big> \,\dx.
\end{align}
\begin{proof}
The proof of this lemma is straightforward. For the sake of completeness, we briefly mention the proof. Note that, in order to prove the claim, we need to compute $\Dif \intTor{ \textbf{q}\cdot {\bf Q} }$. Due to lack of regularity, it is customary to use regularization by convolutions. To that context, let us denote by $(\bm{\rho}_\alpha)$, an approximation to the identity on $\T^3$.
Let $\bm{\varphi}\in L^2_{\text{div}}(\mathbb{T}^3)$, then $\bm{\varphi}_\alpha = \bm{\varphi}*\bm{\rho}_\alpha\,\in C_{\text{div}}^\infty(\mathbb{T}^3)$. Then using $\bm{\varphi}_\alpha$ as a test function in \eqref{eq1}, we obtain $\mathbb{P}$-almost surely, for all $t\in[0,T]$
\begin{align*}
\int_{\mathbb{T}^3}\mathbf{q}_\alpha(t)\cdot\bm{\varphi}\,\D x&=\int_{\mathbb{T}^3}\textbf{q}_\alpha(0)\cdot\bm{\varphi}\,\D x + \int_{0}^t\int_{\mathbb{T}^3}(\mathbf{q}_1)_\alpha:\nabla\bm{\varphi}\, \D x \,\D s+\int_0^t\int_{\mathbb{T}^3}\nabla\bm{\varphi}:\,d\mu_\alpha(x,s) \, \D s \\&\qquad+ \int_{\mathbb{T}^3}\bm{\varphi}\cdot\int_0^tdM_\alpha \, \D x.
\end{align*}
It implies that $\mathbb{P}$-a.s., for all $t\in[0,T]$
\begin{align*}
\mathbf{q}_\alpha(t)&=\textbf{q}_\alpha(0) -\int_0^t\mathcal{P}_H\big(\Div(\mathbf{q}_1)_\alpha \big){\rm d}s-\int_0^t\,\mathcal{P}_H\big(\Div \mu_\alpha\big) \, \D x \, \D s +\int_0^tdM_\alpha \, \D x.
\end{align*}
Now, we can apply It\^o's formula  to the process $t\to  \intTor{ \textbf{q}_r\cdot {\bf Q} }$, then we obtain for all $t\in[0,T],\,\,\mathbb{P}$-a.s
\begin{align*}
\int_{\mathbb{T}^3}\mathbf{q}_\alpha(t)\cdot\mathbf{Q}(t)\dx&=\int_{\mathbb{T}^3}\mathbf{q}_\alpha(0)\cdot\mathbf{Q}(0)\dx + \int_0^t\int_{\mathbb{T}^3} (\mathbf{q}_1)_\alpha:\nabla \mathbf{Q} \dx\,{\rm d}s+\int_0^t\int_{\mathbb{T}^3}\nabla\mathbf{Q}:\,\D\mu_\alpha\,\D s\\
&\qquad+\int_0^t\int_{\mathbb{T}^3}\mathbf{Q}\cdot \D M_\alpha\,\D x+\int_0^t\int_{\mathbb{T}^3}\mathbf{Q}_1\cdot\mathbf{q}_\alpha \,\D x\,\D s + \int_0^t\int_{\mathbb{T}^3}\mathbf{q}_\alpha \cdot\mathbf{Q}_2\,\D W\,\D x\\
&\qquad+\int_{\mathbb{T}^3}\Big<\hspace{-0.14cm}\Big<M_\alpha (t)\,,\, \mathbf{Q(t)}  \Big>\hspace{-0.14cm}\Big> \,\dx\,. 
\end{align*}
By using given hypothese, we can perform the limit $\alpha\to 0$ in above relation to conclude the proof.
\end{proof}
 
\end{Lemma}
\subsection{Relative Energy Inequality (Euler System)}
\label{MEI}

\noindent  It is well-known that the relative energy inequality is very useful for the comparison of a measure valued solution and a smooth given function. To see this, let us first introduce the \textit{relative energy (entropy)} functional in the context of measure-valued solutions to the stochastic incompressible Euler system as
\begin{equation}
\notag
\begin{aligned} 
&\mathrm{F}_{\mathrm{mv}}^1 \left(\vu  \ \Big|{\bf U} \right)(t)
:=
\intTor{\langle {\mathcal{V}^{\omega}_{t,x}}; \frac{1}{2} { |\vu |^2}\rangle } + \mathcal{D}(t)
- \intTor{ \big \langle {\mathcal{V}^{\omega}_{t,x}}; \vu  \big \rangle \,\cdot {\bf U}} +\frac{1}{2}  \intTor{  |{\bf U}|^2 } .
\end{aligned}
\end{equation}
In view of the energy inequality \eqref{energy_001}, it is clear that the above relative energy functional is defined for all $t\in[0,T]\setminus \mathcal{N}$, where the null (i.e., Lebesgue measure zero) set $\mathcal{N}$ may depends on $\omega \in \Omega$. We also define relative energy functional for all $t\in \mathcal{N}$ as follows:
\begin{equation}
	\notag
	\begin{aligned} 
		&\mathrm{F}_{\mathrm{mv}}^2 \left(\vu  \ \Big|{\bf U} \right)(t)
		:=
		\lim_{r\to 0^+}\frac{1}{r}\int_t^{t+r}\bigg[\intTor{\langle {\mathcal{V}^{\omega}_{s,x}}; \frac{1}{2} { |\vu |^2}\rangle }+\mathcal{D}(s)\bigg]ds
		- \intTor{ \big \langle {\mathcal{V}^{\omega}_{t,x}}; \vu  \big \rangle \,\cdot {\bf U}} +\frac{1}{2}  \intTor{  |{\bf U}|^2 }.
	\end{aligned}
\end{equation}
Using above, we define relative energy functional, which is well-defined defined for all $t\in[0,T]$, as follows:
\begin{equation}
	\label{rell}
	\begin{aligned} 
		&\mathrm{F}_{\mathrm{mv}} \left(\vu  \ \Big|{\bf U} \right)(t)
		:=\begin{cases}
			\mathrm{F}_{\mathrm{mv}}^1 (\vu  \ \big|{\bf U})(t), &\text{if $t\in[0,T]\setminus \mathcal{N};$}\\
			\mathrm{F}_{\mathrm{mv}}^2 (\vu  \ \big|{\bf U} )(t),&\text{if $t\in \mathcal{N}$.}
		\end{cases}
	\end{aligned}
\end{equation}
To proceed further, we make use of the relative energy \eqref{rell} to derive the relative energy inequality given by \eqref{relativeEntropy}.
\begin{Proposition}[Relative Energy] 
\label{relen}
Let $\big[ \big(\Omega,\mathbb{F}, (\mathbb{F}_{t})_{t\geq0},\mathbb{P} \big); {\mathcal{V}^{\omega}_{t,x}}, W \big]$ be a dissipative measure-valued martingale solution to the system \eqref{P1}.
Suppose $\mathbf{U}$ be stochastic processes which is adapted to the filtration $(\mathbb{F}_t)_{t\geq0}$ and satisfies
\begin{equation*}
\begin{aligned}
\label{operatorBB}
\mathrm{d}\mathbf{U}  &= \mathbf{U}_1\,\mathrm{d}t  + \mathcal{P}_H\mathbf{U}_2\,\mathrm{d}W,
\end{aligned}
\end{equation*}
with
\begin{align}\label{eq:smooth}
\begin{aligned}\vc{U} \in C([0,T]; C_{\dv}^{1}(\T^3)), \ \quad\text{$\mathbb{P}$-a.s.},\qquad
\E\bigg[ \sup_{t \in [0,T] } \| \vc{U} \|_{L_{\dv}^{2}(\tor)}^2\bigg]\textless\,\infty,
\end{aligned}
\end{align}
Moreover, $\vc{U}$ satisfy
\begin{align}\label{new}
&  \vc{U}_1\in L^2(\Omega;L^2(0,T;L_{\dv}^{2}(\mt))),\quad  \vc{U}_2\in L^2(\Omega;L^2(0,T;L_2(\mathfrak U;L^2(\tor)))),
\end{align}
$$\int_0^T\sum_{k\geq 1}\|\mathcal{P}_H\vc{U}_2(e_k)\|_{L^2(\mathbb{T}^3)}^2\in L^1(\Omega).$$
Then the following \emph{relative energy inequality} holds:
\begin{equation}
\begin{aligned}
\label{relativeEntropy}
&\mathrm{F}_{\mathrm{mv}} \left(\vu  \ \Big|  {\bf U} \right)
(t) \leq
\mathrm{F}_{\mathrm{mv}} \left(\vu  \ \Big| {\bf U} \right)(0) +\mathcal{M}_{RE}(t)  + \int_0^t\mathfrak{R}_{\mathrm{mv}} \big(\vu  \left\vert \right.  \mathbf{U}  \big)(s)\,\mathrm{d}s
\end{aligned}
\end{equation}
$\mathbb P$-almost surely, for all $t \in [0,T]$ with
\begin{align}
\label{remainderRE}
\mathfrak{R}_{\mathrm{mv}} \big({\bf u} \left\vert \right.\mathbf{U}  \big) 
&=
 \int_{\mathbb{T}^3} \left\langle {\mathcal{V}^{\omega}_{t,x}}; {{\bf u}\otimes {\bf u} } \right\rangle: \nabla_x \vu  \,\mathrm{d}x+\int_{\mathbb{T}^3}\left\langle {\mathcal{V}^{\omega}_{t,x}};\bf u\right\rangle\cdot {\bf U}_1\, \D x\, \D t -\int_{\T^3} \nabla_x \mathbf{U} : d\mu_C + \frac12\int_{\T^3} d \mu_D \notag \\
 &\qquad+\frac{1}{2}
 \sum_{k\in\mathbb{N}}
 \int_{\mathbb{T}^3} \big \langle {\mathcal{V}^{\omega}_{t,x}}; \big\vert {\mathbf{G}_k({\bf u})} -\mathbf{U}_2(e_k)  \big\vert^2\big\rangle \,\mathrm{d}x.
\end{align}
Here $\mathcal{M}_{RE}(t)$ is a $\R$-valued square integrable martingale, whose norm depends on the norms of smooth function $\mathbf{U}$ in the aforementioned spaces.
\end{Proposition}

\begin{proof}
We follow the usual strategy and express all the integrals on the right hand side of \eqref{rell}
by making use of the energy inequality \eqref{third condition measure-valued solution} and the field equation \eqref{second condition measure-valued solution}. Therefore, we shall 
make use of It\^o's formula and the energy inequality (\ref{third condition measure-valued solution}) to compute the right hand side of \eqref{rell}.

\medskip 

\noindent {\bf Step 1:} 
In order to compute $\Dif \intTor{ \big \langle {\mathcal{V}^{\omega}_{t,x}}; \vu  \big \rangle \,\cdot {\bf U} }$, we first recall that $\mathbf{q}=\langle {\mathcal{V}^{\omega}_{t,x}}; \vu  \big \rangle $ satisfies hypothesis of Lemma \ref{Ito}. Therefore we can apply the Lemma \ref{Ito} to conclude that $\mathbb{P}$-almost surely
\begin{equation} \label{I1}
\begin{split}
& \Dif \left( \intTor{ \big \langle {\mathcal{V}^{\omega}_{t,x}}; \vu  \big \rangle \cdot {\bf U} } \right) = \intTor{ \left[ \big \langle {\mathcal{V}^{\omega}_{t,x}}; \vu  \big \rangle \cdot {\bf U}_1
+ \left\langle {\mathcal{V}^{\omega}_{t,x}}; {\vu \otimes \vu }\right\rangle: \nabla_x \vu   \right] }  {\rm d}t \\
& \qquad +  \sum_{k\geq1}\intTor{ \mathcal{P}_H{\bf U}_2(e_k) \cdot\mathcal{P}_H \big \langle {\mathcal{V}^{\omega}_{t,x}}; \vc{G}_k (\vu ) \big\rangle}\, {\rm d}t + \int_{\T^3}  \nabla_x {\bf U}: d\mu_C \,dt +  \Dif \mathcal{M}_1,
\end{split}
\end{equation}
where the square integrable martingale $\mathcal{M}_1(t)$ is given by
\[
\mathcal{M}_1(t) = \int_{\T^3} \int_0^t  \vc{U} \, dM^1_{E} \,dx+ \int_0^t \intTor{ \big \langle {\mathcal{V}^{\omega}_{t,x}}; \vu  \big \rangle \cdot \mathcal{P}_H\vc{U}_2 } \,\Dif W
\]
We remark that the item (i) of the Definition \ref{def:dissMartin} is used to identify the cross variation in \eqref{I1}. Indeed, notice that
\begin{align*}
\int_{\mathbb{T}^3}\Big<\hspace{-0.14cm}\Big<f(t),  M^1_{E}(t)  \Big>\hspace{-0.14cm}\Big> dx
&=\int_{\mathbb{T}^3} \sum_{i,j}\Bigg(\sum_{k = 1}^{\infty}  \int_0^t\langle \mathcal{P}_H \left\langle \mathcal{V}^{\omega}_{s,x}; \mathbf{G}_k (\vu )\right\rangle,h_i\rangle \,\langle\mathbb{D}^s f(e_k), h_j\rangle \,ds\Bigg) h_i\otimes h_j \D x.\\&
=\sum_{k\ge\,1}\int_{\T^3}\int_0^t\mathcal{P}_H{\bf U}_2(e_k) \cdot\mathcal{P}_H \big \langle {\mathcal{V}^{\omega}_{t,x}}; \vc{G}_k (\varrho,\vu ) \big\rangle\,\D x {\rm d}t
\end{align*}
\noindent
{\bf Step 2:}
Next, we see that 
\begin{equation} \label{I21}
\begin{split}
\Dif \left( \intTor{ \frac{1}{2} |\vc{U}|^2 } \right) &= \frac{1}{2} \sum_{k\geq1}\intTor{|\mathcal{P}_H {\bf U}_2(e_k)|^2 } \ {\rm d}t + {\rm d}\mathcal{M}_2,
\end{split}
\end{equation}
where
\[
\mathcal{M}_2(t) = \int_0^t \intTor{  \vc{U} \cdot\mathcal{P}_H\vc{U}_2 } \, \Dif W.
\]
\noindent 
{\bf Step 3:}
We have from energy inequality
\medskip
\begin{align}
\label{ph}
\mathrm{E}(t+)\, \leq\, &\mathrm{E}(s-)+ \frac{1}{2} \sum_{k\ge\,1}\int_s^{t} \int_{\T^3} \left\langle \mathcal{V}^{\omega}_{s,x};{|{\tn G_k( u)}|^2} \right\rangle \dx{\rm d}\tau-\frac{1}{2}\sum_{k\ge\,1}\int_s^t\int_{\mathbb{T}^3}\Big(\mathcal{Q}_H\big\langle \mathcal{V}^{\omega}_{s,x};{|{\tn G_k( u)}|}\big\rangle\Big)^2\dx{\rm d}\tau \notag \\&\qquad+ \frac12\int_s^{t} \int_{\T^3} d \mu_D + \int_s^{\tau}  d\mathcal{M}^2_{E}. 			
\end{align}
We now manipulate the product term in the equality \eqref{I1} using properties of projections $\mathcal{P}_H$ and $ \mathcal{Q}_H$. Indeed, note that
\begin{align*}
	&\intTor{ \mathcal{P}_H{\bf U}_2(e_k) \cdot\mathcal{P}_H \big \langle {\mathcal{V}^{\omega}_{t,x}}; \vc{G}_k (\vu ) \big\rangle}\\
	&\qquad =\intTor{ {\bf U}_2(e_k) \cdot\big \langle {\mathcal{V}^{\omega}_{t,x}}; \vc{G}_k (\vu ) \big\rangle}-\intTor{ \mathcal{Q}_H{\bf U}_2(e_k) \cdot\mathcal{Q}_H \big \langle {\mathcal{V}^{\omega}_{t,x}}; \vc{G}_k (\vu ) \big\rangle},
\end{align*}
and
\begin{align*}
\intTor{|\mathcal{P}_H {\bf U}_2(e_k)|^2 }=\intTor{|{\bf U}_2(e_k)|^2 }-\intTor{|\mathcal{Q}_H {\bf U}_2(e_k)|^2 }.
\end{align*}
These properties of projections imply that
\begin{align}\label{rt}
	\begin{aligned}
	&\frac{1}{2} \sum_{k\geq1}\intTor{|\mathcal{P}_H {\bf U}_2(e_k)|^2 } -\sum_{k\geq1}\intTor{ \mathcal{P}_H{\bf U}_2(e_k) \cdot\mathcal{P}_H \big \langle {\mathcal{V}^{\omega}_{t,x}}; \vc{G}_k (\vu ) \big\rangle}\\&\qquad\qquad+ \frac{1}{2} \sum_{k\ge\,1}\int_s^{t} \int_{\T^3} \left\langle \mathcal{V}^{\omega}_{s,x};{|{\vc G_k( u)}|^2} \right\rangle \dx\,{\rm d}\tau-\frac{1}{2}\sum_{k\ge\,1}\int_s^t\int_{\mathbb{T}^3}\Big(\mathcal{Q}_H\big\langle \mathcal{V}^{\omega}_{s,x};{|{\vc G_k( u)}|}\big\rangle\Big)^2\dx\,{\rm d}\tau\\&=\frac{1}{2}\sum_{k\ge\,1}\int_{\mathbb{T}^3} \Big \langle {\mathcal{V}^{\omega}_{t,x}}; \big\vert {\mathbf{G}_k({\bf u})} -\mathbf{U}_2(e_k)  \big\vert^2\Big\rangle \,\mathrm{d}x-\frac{1}{2}\sum_{k\ge\,1}\int_{\mathbb{T}^3}\Big|\mathcal{Q}_H\mathbf{U}_2(e_k)-\mathcal{Q}_H \big \langle {\mathcal{V}^{\omega}_{t,x}}; \vc{G}_k (\vu ) \big\rangle\Big|^2\,\dx\\
	&\le\,\frac{1}{2}\sum_{k\ge\,1}\int_{\mathbb{T}^3} \Big \langle {\mathcal{V}^{\omega}_{t,x}}; \big\vert {\mathbf{G}_k({\bf u})} -\mathbf{U}_2(e_k)  \big\vert^2\Big\rangle \,\mathrm{d}x.
	\end{aligned}
\end{align}
Finally, in view of the above observations given by (\ref{I1})-(\ref{rt}), we can now add the resulting expressions to establish (\ref{relativeEntropy}). Note that the square integrable martingale $\mathcal{M}_{RE}(t)$ is given by $\mathcal{M}_{RE}(t):= \mathcal{M}_1(t) + \mathcal{M}_2(t)+\mathcal{M}_E^2 (t)$.
\end{proof}

\subsection{Proof of Theorem~\ref{Weak-Strong Uniqueness}}
\label{sub01}
In this subsection, we aim at establishing the desired weak (measure-valued)--strong uniqueness principle given by Theorem~\ref{Weak-Strong Uniqueness}. To do so, we need to apply the relative energy inequality \eqref{third condition measure-valued solution} with a specific choice of the smooth given function $\mathbf U=\bar{\bfu}(\cdot\wedge \tau_L)$, where $(\bar{\mathbf{u}},(\tau_L)_{L\in\mathbb{N}},\tau)$ is the unique maximal strong pathwise solution to \eqref{P1}. For technical reason, note that the stopping time $\tau_L$ announces the blow-up and satisfies
\begin{equation*}
\sup_{t\in[0,\tau_L]}\|\bar{\vu}(t)\|_{1,\infty}\geq L\quad \text{on}\quad [\mathfrak{t}<T] ;
\end{equation*}
Furthermore, it is evident that $\bar{\bfu}$ satisfies the equation \eqref{operatorBB}, with
\begin{align*}
\mathbf{U}_1  = \mathcal{P}_H(\bar{\mathbf{u}}\cdot\nabla_x \bar{\mathbf{u}}),
\quad    
\mathbf{U}_2  =\tn{G} (\bar{{\bf u}}).
\end{align*}
Clearly, in view of Theorem \ref{existence of strong solution} and  \eqref{FG1}--\eqref{FG2}, the conditions \eqref{eq:smooth} and \eqref{new} are satisfied for $t\leq \mathfrak t_L$. 
Therefore, the inequality \eqref{relativeEntropy} holds, and we can also deduce from \eqref{remainderRE} that
\begin{equation}
\begin{aligned}
\label{relativeEntropy1}
&\mathrm{F}_{\mathrm{mv}}\left(\vu \Big| \bar{\mathbf{u}}\right)  
(t \wedge \tau_L) \leq
\mathrm{F}_{\mathrm{mv}} \left(\vu  \Big|\bar{\mathbf{u}}\right)(0) + M_{RE}(t \wedge \tau_L)  + \int_0^{t \wedge \tau_L} \mathfrak{R}_{\mathrm{mv}} \big({\bf u} \left\vert \right. \bar{\mathbf{u}}  \big) (s)\,\mathrm{d}s,
\end{aligned}
\end{equation}
holds for each $L\in\mathbb{N}$, for all $t\in[0,T]$, $\mathbb{P}$-almost surely. Here after manipulating terms in \eqref{remainderRE}, as in \cite{emil}, we obtain
\begin{equation}
\begin{aligned}
\notag
\mathfrak{R}_{\mathrm{mv}}\left({\bf u}  \Big|\bar{\mathbf{u}}\right) 
=&\,  \int_{\mathbb{T}^3} \left\langle {\mathcal{V}^{\omega}_{t,x}}; \left|{(\vu - \bar{\mathbf{u}})\otimes ( \bar{\mathbf{u}}-\vu ) }\right| \right\rangle |\nabla_x \bar{\mathbf{u}}| \,dx +
\frac{1}{2}
\sum_{k\in\mathbb{N}}
\int_{\mathbb{T}^3} \big \langle {\mathcal{V}^{\omega}_{t,x}}; \big\vert {\mathbf{G}_k(\mathbf{u})}  -{\mathbf{G}_k(\bar{\mathbf{u}})} \big\vert^2 \big \rangle \,\mathrm{d}x \\
&\qquad+  \int_{\T^3} |\nabla_x \bar{\mathbf{u}}|\cdot d|\mu_C|+  \int_{\T^3} d|\mu_D|.
\end{aligned}
\end{equation}
Since $\|\mathbf{\bar{u}}\|_{W^{1,\infty}(\mathbb{T}^3)}\le\,c(L)$ for $t\,\le\,\tau_L$, we can control the terms $|\nabla_x \bar{\mathbf{u}}|$ by some constant. It is also clear that
\begin{equation*}
\left| {(\vu - \bar{\mathbf{u}}) \otimes ( \bar{\mathbf{u}}-\vu )} \right| \leq |\vu -\bar{\mathbf{u}}|^2,
\end{equation*}
and
\begin{equation}
\begin{aligned}
\notag
\sum_{k \ge 1}\big\vert {\mathbf{G}_k(\mathbf{u})}  -{\mathbf{G}_k(\bar{\mathbf{u}})} \big\vert^2\leq D_1|\mathbf{u}-\bar{\mathbf{u}}|^2.
\end{aligned}
\end{equation}
Finally, we also see that
\begin{align*}
\frac{1}{2}
\sum_{k\in\mathbb{N}}
\int_{\mathbb{T}^3} \big \langle {\mathcal{V}^{\omega}_{t,x}}; \big\vert {\mathbf{G}_k(\mathbf{u})}  -{\mathbf{G}_k(\bar{\mathbf{u}})} \big\vert^2 \big \rangle \,\mathrm{d}x
\leq c(L)\,
\mathrm{F}^1_{\mathrm{mv}} \big( {\mathbf{u}}\left\vert \right. \bar{\mathbf{u}}  \big).
\end{align*}
Collecting all the above estimates and using the item $(h)$ of Definition \ref{def:dissMartin}, we conclude that for all $t\in[0,T]$, $\mathbb{P}$-a.s.
\begin{equation}
\begin{aligned}
\label{r4est}
\int_0^{t \wedge \tau_L}
\mathfrak{R}_{\mathrm{mv}} \big({\bf u} \left\vert \right.\bar{\mathbf{u}}  \big)  \,\mathrm{d}s
\leq c(L)\,\int_0^{t \wedge \tau_L}
\Big(\mathrm{F}_{\mathrm{mv}} \big({\bf u} \left\vert \right. \bar{\mathbf{u}}  \big) 
(s)\Big)\,\mathrm{d}s .
\end{aligned}
\end{equation}
We now combine \eqref{r4est} and \eqref{relativeEntropy1}, and apply classical Gronwall's lemma, to obtain for all $t\in [0,T]$
\begin{align}
\nonumber
\mathbb{E}\,  \Big[\mathrm{F}_{\mathrm{mv}} \big({\mathbf{u}}\left\vert \right. \bar{\mathbf{u}}  \big)  
(t \wedge \tau_L)\Big] 
\leq c(L)\,
\mathbb{E}\,\Big[\mathrm{F}_{\mathrm{mv}} \big({\mathbf{u}}\left\vert \right. \bar{\mathbf{u}}  \big)(0)\Big].
\end{align}
We recall that
\begin{equation*}
\begin{aligned}
\mathrm{F}_{\mathrm{mv}} \big(\varrho ,{\mathbf{m}}\left\vert \right. \bar{\varrho}, \bar{\mathbf{u}}  \big)  (0) 
&=
\int_{\mathbb{T}^3} \Big \langle {\mathcal{V}^{\omega}_{0,x}}; \frac{1}{2}\varrho_{0}\big\vert \mathbf{u}_{0} - \bar{\mathbf{u}}_{0} \big\vert^2 \Big \rangle \,\mathrm{d}x,
\end{aligned}
\end{equation*}
which, by assumption, vanishes in expectation. Therefore, we conclude that
$$
\mathbb{E}\,  \Big[\mathrm{F}_{\mathrm{mv}} \big({\mathbf{u}}\left\vert \right. \bar{\mathbf{u}}  \big)  
(t \wedge \tau_L)\Big]=0,\,\,\text{for all}\,\,\,t\in[0,T].
$$
This also implies that 
$$
\lim_{r\to 0^+}\frac{1}{r}\int_{t}^{t+r}\mathbb{E}\,  \Big[\mathrm{F}_{\mathrm{mv}} \big({\mathbf{u}}\left\vert \right. \bar{\mathbf{u}}  \big)  
(s \wedge \tau_L)\Big]\, \D s=0.
$$
Keeping in mind a priori estimates, a standard Lebesgue point argument in combination with classical Fubini's theorem reveals that for a.e. $t\in[0,T]$,
$$\mathbb{E}\,  \Big[\mathrm{F}_{\mathrm{mv}}^1 \big({\mathbf{u}}\left\vert \right. \bar{\mathbf{u}}  \big)  
(t \wedge \tau_L)\Big]=0.$$
But since the defect measure $\mathcal{D} \geq 0$, above equality implies for a.e. $t\in[0,T]$, $\mathbb{P}$-almost surely
 $$\mathcal{D}(t\wedge \tau_L)=0, \,\,\text{and}\,\,\, \bar{{\bf u}}(x,t\wedge\tau_L)={\bf u}(x, t\wedge\tau_L),\,\,\text{ for a.e.}\,\,x\in \mathbb{T}^3.
 $$


%


 \section{Weak-Strong Uniqueness for Navier--Stokes System} 
\label{NSE}

Let $\bf u$ and $\bf U$ be two finite energy weak martingale solutions to \eqref{P1NS}, with same inital data $u_0$, defined on the same stochastic basis. The commonly used form of the \textit{relative energy} functional in the context of weak  solutions to the incompressible Naiver stokes system reads
\begin{equation}
\label{rell1}
\begin{aligned} 
\mathrm{F}^{\mathrm{NS}}_{\mathrm{mv}} \left(\vu  \, \Big|\,{\bf U} \right):=
\intTor{\frac{1}{2} { |\vu |^2} }- \intTor{ \vu   \,\cdot {\bf U}} +\frac{1}{2}  \intTor{  |{\bf U}|^2 }.
\end{aligned}
\end{equation}
The proof of (weak-strong) uniqueness for finite energy weak martingale solutions to \eqref{P1NS} essentially uses similar arguments, as depicted in Section~\ref{proof2}. However, the main difficulty lies in the successful identification of cross variation of two martingale solutions. Indeed, the regularity of finite energy weak martinagle solutions is not enough to identify the cross variation between them, and requires one solution to be more regular. In what follows, we start with the following lemma, whose proof is a simple consequence of the H\"older and Sobolev inequalities, see \cite[Lemma 2.4]{Masuda}.
\begin{Lemma}\label{inequality1}
Let $r,s$ satisfy
$$\frac{3}{s}+\frac{2}{r}=1,\,\,\,\,s\in(3,\infty)$$
and let ${\bf u,w}\in L^2((0,T);H^1_{\dv}(\mathbb{T}^3))$, and $\vu \in L^r(0,T;L^s(\mathbb{T}^3))$. Then
\begin{align*}
\bigg|\int_0^T\langle \textbf{v}\cdot\nabla  \textbf{w,u}\rangle\bigg|\le\,C\bigg(\int_0^T\|\nabla \textbf{w}\|&_{L^2(\mathbb{T}^3)}^2dt\bigg)^{\frac12}\bigg(\int_0^T\|\nabla \textbf{v}\|_{L^2(\mathbb{T}^3)}^2\bigg)^{\frac{3}{2s}} \bigg(\int_0^T\|\vu \|_{L^s(\mathbb{T}^3)}^r\|\textbf{v}\|_{L^2(\mathbb{T}^3)}^2dt\bigg)^{\frac1r}.
\end{align*}
\end{Lemma}
To make use of the above inequality, let us assume that two finite energy weak martingale solutions
${\bf u,U}\,\in L^{\infty}((0,T);L^2(\mathbb{T}^3))\cap L^2((0,T);H^1(\mathbb{T}^3))$ be diveregence free. Assume, in addition, that ${\bf U}\in L^r((0,T);L^s(\mathbb{T}^3)).$ Then, in view of the above Lemma~\ref{inequality1}, for every $\tau\in(0,T)$
\begin{align}\label{inequality}
\bigg|\int_0^{\tau}\int_{\mathbb{T}^3}& \nabla{\bf(u-U):((u-U)}\otimes{\bf U}) \, \D x \, \D  t\bigg|\\\notag&\le\,C\bigg(\int_0^{\tau}\int_{\mathbb{T}^3}|\nabla{\bf(u-U)}|^2\, \D x\bigg)^{1-1/r}\bigg(\int_0^\tau\|{\bf U}\|_{L^s(\mathbb{T}^3))}^r\int_{\mathbb{T}^3}{\bf|u-U|^2}\,\D x\, \D t\bigg)^{1/r}.
\end{align}
In order to calculate the evolution equation satisfies by the second term of the relative energy \eqref{rell1}, we need to apply It\^o product rule. To do so, we first regularize \eqref{e1q:energy1} (with $\varepsilon =1$), for both solutions $\vu$ and $\bf U$, by taking a spatial convolution with a suitable family of regularizing kernels. We denote by ${\bf v}_r$, the regularization of $\bf v$. For a test fuction $\bm{\varphi}\in L^2_{\text{div}}(\mathbb{T}^3)$, we have $\bm{\varphi_r}\in C_{\text{div}}^{\infty}(\mathbb{T}^3)$. For both weak solutions $\bf u$ and $\bf U$, we may write equations
\begin{align*}
\langle \mathbf{u}(t), \bm\varphi_r\rangle & = \langle \mathbf{u} (0), \bm{\varphi_r}\rangle +\int_0^{t}\langle \mathbf{u}\otimes\mathbf{u}(s),\nabla_x\bm{\varphi_r}\rangle\,\mathrm{d}s
-\int_0^{t} \langle \nabla_x \mathbf{u}_(s)\,  \nabla_x\bm{\varphi_r}\rangle\mathrm{d}s 
+\int_0^{t}\langle\mathcal{G}(\mathbf{u}), \bm{\varphi_r}\rangle\,\mathrm{d}W,\\
\langle \mathbf{U}(t), \bm\varphi_r\rangle & = \langle \mathbf{U} (0), \bm{\varphi_r}\rangle +\int_0^{t}\langle \mathbf{U}\otimes\mathbf{U}(s),\nabla_x\bm{\varphi_r}\rangle\,\mathrm{d}s
-\int_0^{t} \langle \nabla_x \mathbf{U}(s)\,,  \nabla_x\bm{\varphi_r}\rangle\mathrm{d}s 
+\int_0^{t}\langle\mathcal{G}(\mathbf{U}), \bm{\varphi_r}\rangle\,\mathrm{d}W.
\end{align*}
After shiftting regularizing kernel from test fucntion to solutions term, we obtain
\begin{align*}
\mathbf{u}_r(t)&=\mathbf{u}_r (0) -\int_0^{t}\mathcal{P}_H( \Div(\mathbf{u}\otimes\mathbf{u})_r)\mathrm{d}s
+\int_0^{t}\Delta \mathbf{u}_r(s)\mathrm{d}s	+\int_0^{t}\mathcal{P}_H(\mathcal{G}(\mathbf{u})_r)\mathrm{d}W,\\
\mathbf{U}_r(t)&=\mathbf{U}_r (0) -\int_0^{t}\mathcal{P}_H( \Div(\mathbf{U}\otimes\mathbf{U})_r)\mathrm{d}s
+\int_0^{t}\Delta \mathbf{U}_r(s)\mathrm{d}s	+\int_0^{t}\mathcal{P}_H(\mathcal{G}(\mathbf{U})_r)\mathrm{d}W,
\end{align*}
in $(L^2_{\mathrm{div}}(\mathbb{T}^3))'.$
We can now apply classical It\^o's formula to the process 
$t\,\to\,\int_{\mathbb{T}^3}\vu_r \cdot\bf{U}_r \,\D x$, to obtain $\p$-almost surely
\begin{equation} \label{r}
\begin{split}
\Dif \left( \intTor{ {\vu _r  \cdot {\bf {U}}_r }} \right) &= \intTor{ \Big[   -{{\vu }_r  \cdot \Div({\bf U\otimes \bf U})}_r+  ({{\vu \otimes \vu }})_r: \nabla_x {\vu }_r  \Big] } \, {\rm d}t \\
& \quad +  \sum_{k\geq1}\intTor{\mathcal{P}_H({\tn G}_k({\bf U})_r) \cdot \mathcal{P}_H(\vc{G}_k ({\vu )_r)} }\, {\rm d}t -2\int_{\T^3}  \nabla_x {\bf U}_r: \nabla_x {\bf u}_r \, \D x \, \D t\\
&\qquad +\int_0^t \intTor{ ({\vu }_r  \cdot {\tn G(\bf U)}_r + {\bf U} _r\cdot{\tn{G}}({\bf{u}}_r) )}\,\Dif W,\\
\end{split}
\end{equation}
We now wish to let $r\to 0$ in the above relation \eqref{r}. The main difficulty in passing to the limits in the parameter $r$ stems from the nonlinear terms, treatment of others terms is classical. Indeed, we may apply the classical BDG inequality, with the help of a priori estimate and given conditions for noise coefficients, to handle stochastic terms appeared in \eqref{r}. In what follows, we show, with the help of Lemma \eqref{inequality1} and extra regularity of ${\bf U}$, we can pass to the limit in the nonlinear terms. Observe that
\begin{align*} 
&\int_0^t\int_{\mathbb{T}^3}{\bf{u}}_r\cdot \Div({\bf{U}}\otimes{\bf{U}})_r \,\D s \,\D x-\int_0^t\int_{\mathbb{T}^3}{\bf{u}}\cdot \Div({\bf{U}}\otimes{\bf{U}})\,\D s \,\D x\\
&\qquad=\bigg(\int_0^t\int_{\mathbb{T}^3}{\bf{u}}_r\cdot \Div({\bf{U}}\otimes{\bf{U}})_r \,\D s \,\D x-\int_0^t\int_{\mathbb{T}^3}{\bf{u}}_r\cdot \Div({\bf{U}}\otimes{\bf{U}}) \,\D s \,\D x\bigg)\\
&\qquad\qquad+\bigg(\int_0^t\int_{\mathbb{T}^3}{\bf{u}}_r\cdot \Div({\bf{U}}\otimes{\bf{U}})_r \,\D s \,\D x-\int_0^t\int_{\mathbb{T}^3}{\bf{u}}\cdot \Div({\bf{U}}\otimes{\bf{U}}) \,\D s \,\D x\bigg)=:I_1^r+I_2^r,
\end{align*}
For the second nonlinear term,
\begin{align*} 
&\int_0^t\int_{\mathbb{T}^3}{\bf{U}}_r\cdot \Div({\bf{u}}\otimes{\bf{u}})_r \,\D s \,\D x-\int_0^t\int_{\mathbb{T}^3}{\bf{U}}\cdot \Div({\bf{u}}\otimes{\bf{u}}) \,\D s \,\D x\\
&\qquad=\bigg(\int_0^t\int_{\mathbb{T}^3}{\bf{U}}_r\cdot \Div({\bf{u}}\otimes{\bf{u}})_r \,\D s \,\D x-\int_0^t\int_{\mathbb{T}^3}{\bf{U}}_r\cdot \Div({\bf{u}}\otimes{\bf{u}}) \,\D s \,\D x\bigg)\\
&\qquad\qquad+\bigg(\int_0^t\int_{\mathbb{T}^3}{\bf{U}}_r\cdot \Div({\bf{u}}\otimes{\bf{u}})_r \,\D s \,\D x-\int_0^t\int_{\mathbb{T}^3}{\bf{U}}\cdot \Div({\bf{u}}\otimes{\bf{u}}) \,\D s \,\D x\bigg) =:J_1^r+J_2^r,
\end{align*}
First note that
\begin{align*}
|I_1^r|\,&\le\,C\|({\bf{U}}\otimes{\bf U})_r-({\bf U}\otimes{\bf U})\|_{L^1([0,T];L^1(\mathbb{T}^3))},\\
|J_1^r|&\le\,C\,\|({\bf{u}}\otimes{\bf u})_r-({\bf u}\otimes{\bf u})\|_{L^1([0,T];L^1(\mathbb{T}^3))}.
\end{align*}
Moreover, thanks to Lemma~\ref{inequality1}, we conclude
\begin{align*}
|I_2^r|\le\,C\bigg(\int_0^t\|\nabla({\bf{u}}_r-{\bf{u}})\|_{L^2(\mathbb{T}^3)}^2d\tau\bigg)^{1/2}, \quad 
|J_2^r|\le\,C\|{\bf{U}}_r-{\bf{U}}\|_{L^r([0,T];L^s(\mathbb{T}^3))},
\end{align*}
where the constant $C$ depends on $\bf{U}$ and $\bf{u}$ only. Therefore, a simple property of convolution reveals that
\begin{align}
\label{limit}
\lim_{r\to\,0}I_1^r=\lim_{r\to\,0}I_2^r=\lim_{r\to\,0}J_1^r=\lim_{r\to\,0}J_2^r=0.
\end{align}
Finally, letting $r\to 0$ in \eqref{r}, and using \eqref{limit}, we obtain
\begin{equation} 
\label{l1122}
\begin{split}
\Dif \left( \intTor{ \vu   \cdot {\bf U} } \right) &= \intTor{ \left[   -\vu   \cdot \Div({\bf U}\otimes {\bf U})
+  {\vu \otimes \vu }: \nabla_x \vu   \right] }  \,{\rm d}t \\
& \qquad +  \sum_{k\geq1}\intTor{ {\textbf{G}}_k({\bf U}) \cdot {\textbf{G}}_k ({\vu }) }\, {\rm d}t - \int_{\T^3}  \nabla_x {\bf U}: \nabla_x {\bf u} \, \D x \,\D t +  \Dif \mathcal{Q}_1,
\end{split}
\end{equation}
where $\mathcal{Q}_1(t)$ is a square integrable martingale given by
\[
\mathcal{Q}_1(t):=\int_0^t \int_{\mathbb{T}^3}{(\vu   \cdot \tn G(\bf U)+\bm U\cdot\tn{G}(\vu ) } \,\D x\, \Dif W.
\]

\subsection{Proof of Theorem~\ref{Weak-Strong Uniqueness1}}
We closely follow the strategy depicted in Subsection~\ref{sub01}. To proceed, we first introduce a stopping time
$$\kappa_L:=\text{inf}\,\Big\{t\in(0,T)\Big|\,\|{\bf U}(t)\|_{L^s(\mathbb{T}^3)}\,\ge\,L\Big\}$$
Since $\mathbb{E}\Big[\sup_{t\in[0,T]}\|{\bf U}\|_{L^s(\mathbb{T}^N)}\Big]\,\textless\,\infty$ by assumption, we have
$$\mathbb{P}[\kappa_L\,\textless\,T]\,\le\,\mathbb{P}\Big[\sup_{t \in [0,T]}\|{\bf U}\|_{L^s(\mathbb{T}^3)}\,\ge\,L\Big]\,\le\,\frac{1}{L}\mathbb{E}\Big[\sup_{t \in [0,T]}\|{\bf U}\|_{L^s(\mathbb{T}^3)}\Big]\,\to\,0,$$
$$\mathbb{P}\bigg[\lim_{L\to\infty}\kappa_L=T\bigg]=1.$$
Therefore, it is enough to show the result for a fixed $L$.
We now make use of the relative energy \eqref{rell1}, the energy inequality \eqref{eq:apriorivarepsilon1}, and \eqref{l1122}, for both solutions, to conclude that for all $t\in[0,T]$, $\mathbb{P}$-a.s.
\begin{align*}
&\mathrm{F}^{\mathrm{NS}}_{\mathrm{mv}} \left(\vu  \Big|{\bf U}\right)(t \wedge \kappa_L)+\int_0^{t \wedge \kappa_L} \intTor{|\nabla_x{\bf (u-U)}|^2}\\&\qquad =\,\frac{1}{2}\intTor{|{\bf u}(t \wedge \kappa_L)|^2}+\frac{1}{2}\intTor{ |{\bf U}(t \wedge \kappa_L)|^2}+ \int_0^{t \wedge \kappa_L}\intTor{|\nabla_x {\bf u}|^2}dt\\&\qquad\qquad+\int_0^{t \wedge \kappa_L}\intTor{|\nabla_x{\bf U}|^2}dt-\intTor{{\bf u}(t \wedge \kappa_L)\cdot{\bf U}(t \wedge \kappa_L)}-2\int_0^{t \wedge \kappa_L}\intTor{\nabla_x{\bf u}:\nabla_x{\bf U}}dt\\
&\qquad \le \mathrm{F}^{\mathrm{NS}}_{\mathrm{mv}}\left({\bf u}(0)\Big|{\bf U}(0)\right)+\int_0^{t \wedge \kappa_L}\intTor{\nabla_x{\bf u}:({\bf(u-U)}\otimes {\bf U})}dt+\mathcal{Q}_{RE}(t \wedge \kappa_L)\\&\qquad\qquad+\frac{1}{2}\int_0^{t \wedge \kappa_L}\intTor{|\mathcal{P}_H({\tn{G}}({\bf u})-{\tn G}({\bf U}))|^2}dt\\
&\qquad=\,\mathrm{F}^{\mathrm{NS}}_{\mathrm{mv}}\left({\bf u(0)}\Big|{\bf U}(0)\right)+\int_0^{t \wedge \kappa_L}\intTor{\nabla_x ({\bf u-U}):(({\bf u-U})\otimes {\bf U})}+\mathcal{Q}_{RE}(t \wedge \kappa_L)\\&\qquad\qquad+\frac{1}{2}\int_0^{t \wedge \kappa_L}\intTor{|\mathcal{P}_H({\tn{G}}({\bf u})-{\tn G}({\bf U}))|^2}dt\\&\qquad
\le\,\mathrm{F}^{\mathrm{NS}}_{\mathrm{mv}}\left({\bf u}(0)\Big|{\bf U}(0)\right)+C\,\bigg(\int_0^{t \wedge \kappa_L}\int_{\mathbb{T}^N}|\nabla({\bf u-U})|^2\bigg)^{1-1/r}\bigg(\int_0^{t \wedge \kappa_L}\|{\bf U}\|_{L^s}^r\int_{\mathbb{T}^3}|{\bf u-U}|^2\,dxdt\bigg)^{1/r}\\&\qquad\qquad+\mathcal{Q}_{RE}(t \wedge \kappa_L)+\frac{1}{2}\int_0^{t \wedge \kappa_L}\intTor{|{\tn{G}}({\bf u})-{\tn G}({\bf U})|^2}dt\\
&\qquad\le\,\mathrm{F}^{\mathrm{NS}}_{\mathrm{mv}}\left({\bf u}(0)\Big|{\bf U}(0)\right)+\int_0^{t \wedge \kappa_L}\intTor{|\nabla_x{\bf (u-U)}|^2}+C\,\bigg(\int_0^{t \wedge \kappa_L}\|{\bf U}\|_{L^s}^r\int_{\mathbb{T}^3}|{\bf u-U}|^2\,dxdt\bigg)\\&\qquad\qquad+\mathcal{Q}_{RE}(t \wedge \kappa_L)+\frac{D_ 1}{2}\int_0^{t \wedge \kappa_L}\intTor{{\bf |u-U|}^2}dt\\
&\qquad\le\,\mathrm{F}^{\mathrm{NS}}_{\mathrm{mv}}\left ({\bf u}(0)\Big|{\bf U}(0)\right)+\int_0^{t \wedge \kappa_L}\intTor{|\nabla_x({\bf u-U})|^2}+C(L)\bigg(\int_0^{t \wedge \kappa_L}\int_{\mathbb{T}^3}|{\bf u-U}|^2\,dxdt\bigg)\\&\qquad\qquad+\mathcal{Q}_{RE}(t \wedge \kappa_L)+\frac{D_ 1}{2}\int_0^{t \wedge \kappa_L}\intTor{{\bf |u-U|}^2}dt,
\end{align*}
where
\begin{align*}
\mathcal{Q}_{RE}(t)=\mathcal{Q}_1(t)+\int_0^t\int_{\mathbb{T}^3}\Big(\vu \, \mathcal{P}_H \tn G(\vu )+{\bf U}\, \mathcal{P}_H \tn G({\bf U} )\Big)\, \D x\, \D W.
\end{align*}
Note that, in the above calculations, we have used the estimate \eqref{inequality}, and classical Young's inequality
\begin{align*}
x y\,\le\,\frac{\gamma x^a}{a}+\frac{y^b}{b\gamma^{b/a}}
\end{align*}
for any $\gamma\,\textgreater 0,\,\frac{1}{a}+\frac{1}{b}=1$, and $x,y\,\textgreater 0$. Therefore, we obtain for all $t\in[0,T]$, $\mathbb{P}$-a.s.
\begin{equation}
\mathrm{F}^{\mathrm{NS}}_{\mathrm{mv}} \left(\vu  \ \Big|{\bf U} \right)(t \wedge \kappa_L)\le\,\mathrm{F}^{\mathrm{NS}}_{\mathrm{mv}}\left({\bf u}(0)\Big|{\bf U}(0)\right)+C(L)\int_0^{t \wedge \kappa_L}{\mathrm{F}^{\mathrm{NS}}_{\mathrm{mv}} \left(\vu  \ \Big|{\bf U} \right)(s)}ds+\mathcal{Q}_{RE}(t \wedge \kappa_L).\notag
\end{equation}
After taking expectation both side, we have
\begin{equation}\notag
\mathbb{E}\bigg[\mathrm{F}^{\mathrm{NS}}_{\mathrm{mv}} \left(\vu  \ \Big|{\bf U} \right)(t \wedge \kappa_L)\bigg]\le\,\mathbb{E}\bigg[\mathrm{F}^{\mathrm{NS}}_{\mathrm{mv}}\left({\bf u}(0)\Big|{\bf U}(0)\right )+C(L)\int_0^{t \wedge \kappa_L}{\mathrm{F}^{\mathrm{NS}}_{\mathrm{mv}} \left(\vu  \ \Big|{\bf U} \right)(s)}ds\bigg].
\end{equation}
Since 
$\mathbb{E}\bigg[\mathrm{F}^{\mathrm{NS}}_{\mathrm{mv}}\left(\bf u(0)\Big|\bf U(0)\right)\bigg]=0$, 
we may use Gronwall's inequality to conclude that
\begin{align*}\mathbb{E}\bigg[\mathrm{F}^{\mathrm{NS}}_{\mathrm{mv}} \left(\vu  \ \Big|{\bf U} \right)(t \wedge \kappa_L)\bigg]=0.
\end{align*}
We can now pass to the limit as $L\to\infty$ to conclude that 
\begin{align*}\mathbb{E}\bigg[\mathrm{F}^{\mathrm{NS}}_{\mathrm{mv}} \left(\vu  \ \Big|{\bf U} \right)(t)\bigg]=0, \,\,\text{for all}\,\,t\in[0,T].
\end{align*}
This implies that $\mathbb{P}$-almost surely, for all $t\in[0,T]$, 
$$
\vu (x,t)={\bf U}(x,t),\,\,\,\text{for a.e.}\,\,x\in\mathbb{T}^3.
$$
This finishes the proof of the theorem.

\section*{Acknowledgements}
U.K. acknowledges the support of the Department of Atomic Energy,  Government of India, under project no.$12$-R$\&$D-TFR-$5.01$-$0520$, and India SERB Matrics grant MTR/$2017/000002$.


\begin{thebibliography}{10}

\bibitem{Balder}
E. J. Balder:
\newblock Lectures on Young measure theory and its applications in economics,
\newblock{\em Rend. Iftit. Mat. Univ. Trieste}, 31 (Suppl. 1), 1-69, 2001.



\bibitem{BhKoleyVa}
N. Bhauryal, U. Koley, G. Vallet:
\newblock The Cauchy problem for a fractional conservation laws driven by L\'{e}vy noise.
\newblock {\em Stochastic Processes and their applications}, 130(9), 5310-5365, 2020. https://doi.org/10.1016/j.spa.2020.03.009



\bibitem{BhKoleyVa_01}
N. Bhauryal, U. Koley, G. Vallet:
\newblock A fractional degenerate parabolic-hyperbolic Cauchy problem with noise.
\newblock {\em Submitted}, https://arxiv.org/pdf/2008.03141.pdf



\bibitem{BisKoleyMaj}
I. H. Biswas, U.~ Koley, and A.~ K. Majee:
\newblock Continuous dependence estimate for conservation laws with  L\'{e}vy noise.
\newblock{\em J. Diff. Equ.}, 259, 4683-4706, 2015.



\bibitem{BrMo}
D.~Breit, T.~C.~Moyo:
\newblock Dissipative solutions to the stochastic Euler equations.
\newblock {\em Submitted,} https://arxiv.org/abs/2008.09517.



\bibitem{m}
D.~Breit, E,~Feiresl, M.~Hofmanova:
\newblock Stochastically forced compressible fluid flows. 
\newblock{\em De Gruyter Series in Applied and Numerical Mathematics. De Gruyter, Berlin/Munich/Boston, (2018).}



\bibitem{BrFeHobook}
D. Breit, E. Feireisl, M. Hofmanov\'a:
\newblock {\em Stochastically forced compressible fluid flows.} 
\newblock {\em De Gruyter Series in Applied and Numerical Mathematics.} De Gruyter, Berlin/Munich/Boston,  (2018).




\bibitem{brenier}
Y.~Brenier, C.~De~Lellis and L.~Sz{\'e}kelyhidi, Jr.:
\newblock Weak-strong uniqueness for measure-valued solutions.
\newblock {\em Comm. Math. Phys.}, 305 (2), 351--361, 2011.




\bibitem{buck01}
T.~Buckmaster, and V.~Vicol:
\newblock Convex integration and phenomenologies in turbulence.
\newblock {\em EMS Surveys in Mathematical Sciences.}, 6, no. 1/2, 173--263, 2019.



\bibitem{buck02}
T.~Buckmaster, and V.~Vicol:
\newblock Non-uniqueness of weak solutions to the Navier--Stokes equation.
\newblock {\em Ana. of Math.}, 189(1), 101--144, 2019.



\bibitem{K1}
A.~Chaudhary, and U.~Koley:
\newblock A convergent finite volume scheme for stochastic compressible barotropic Euler equations, 
\newblock Submitted.




\bibitem{chio01}
E.~Chiodaroli, O.~Kreml, V.~M\'acha, and S. Schwarzacher:
\newblock Non-uniqueness of admissible weak solutions to the compressible Euler equations with smooth initial data.
\newblock Arxiv Preprint Series, arXiv 1812.09917v1, 2019.



\bibitem{prato}
G.~Da Prato, J.~Zabczyk:
\newblock Stochastic equations in infinite dimensions. 
\newblock{\em Cambridge University Press., Cambridge, 1992, 1999.}



\bibitem{Debussche}
A.~Debussche, N.~Glatt-Holtz, R.~Temam:
\newblock Local martingale and pathwise solutions for an abastract fluids model.
\newblock {\em Physica D.}, 240(14-15): 1123--1144,
1999.



\bibitem{DelSze3}
C.~De~Lellis and L.~Sz{\'e}kelyhidi, Jr.:
\newblock On admissibility criteria for weak solutions of the {E}uler
equations.
\newblock {\em Arch. Ration. Mech. Anal.}, {\bf 195}(1):225--260, 2010.



\bibitem{DelSze}
C.~De~Lellis, L.~Sz\'ekelyhidi, Jr.:
\newblock The {$h$}-principle and the equations of fluid dynamics.
\newblock {\em Bull. Amer. Math. Soc. (N.S.)}, 49(3):347--375, 2012.



\bibitem{DL} 
R.~J.~DiPerna: 
\newblock Measure valued solutions to conservation laws.
\newblock {\em Arch. Rational Mech. Anal.} 88(3), 223--270, 1985.




\bibitem{Majda}
R.J.~Diperna and A.J.~Majda:
\newblock Oscillations and concentrations in weak solution of the incompressible fluid equations,
\newblock{\em Comm. Math. Physc. 108(4):667-689,1987.}



\bibitem{Fei01}
E.~Feireisl, P.~Gwiazda, A.~\'{S}wierczewska-Gwiazda, and E. ~Wiedemann:
\newblock Dissipative measure-valued solutions to the compressible Navier-Stokes system.
\newblock {\em Calc. Var. Partial Differential Equations.}, 55(6): Art No. 141, 2016.



\bibitem{FlandoliGatarek}
F.~Flandoli, D.~Gatarek:
\newblock Martingale and stationary solutions for stochastic Navier-Stokes equations.
\newblock {\em Probab. Th. Rel. Fields}, 102: 367-391, 1995.



\bibitem{GHVic}
N.~E. Glatt-Holtz, V.~C. Vicol:
\newblock Local and global existence of smooth solutions for the stochastic
{E}uler equations with multiplicative noise.
\newblock {\em Ann. Probab.}, 42(1):80--145, 2014.




\bibitem{MKS01}
M.~Hofmanova, U.~Koley, and U. Sarkar:
\newblock Measure-valued solutions to the stochastic compressible Euler equations and incompressible limits.
\newblock{\em Under Preparation.}




\bibitem{hof}
M.~Hofmanov\'a, R.~Zhu, X.~Zhu:
\newblock On ill- and well-posedness of dissipative martingale solutions to stochastic $3$D Euler equations.
\newblock {\em https://arxiv.org/pdf/2009.09552.pdf}, 2020.




\bibitem{Jakubowski}
A. ~Jakubowski:
\newblock The almost sure Skorokhod representation for subsequences in nonmetric spaces.
\newblock Theory Probab. Appl. Vol. 42, No. 1 (1998) 164-174.


\bibitem{Kim2}
J.~U. Kim:
\newblock Measure valued solutions to the stochastic {E}uler equations in $\R^d$.
\newblock {\em Stoch PDE: Anal Comp.}, 3: 531--569, 2015.



\bibitem{Koley1}
U. Koley, A. K. Majee, and G. Vallet:
\newblock A finite difference scheme for conservation laws driven by L\'evy noise. 
\newblock{ IMA J. Numer. Anal.}, 38(2), 998–1050, 2018. 



\bibitem{Koley2}
U. Koley, A. K. Majee, and G. Vallet:
\newblock Continuous dependence estimate for a degenerate parabolic-hyperbolic equation with L\'evy noise.
\newblock{ Stoch. Partial Differ. Equ. Anal. Comput.}, 5 (2), 145–191, 2017. 




\bibitem{lions}
P.~L.~Lions:
\newblock Mathematical topics in fluid mechanics.
\newblock {\em vol. 1. Incompressible models}, Oxford Lecture Ser. Math. Appl., vol. 3, Clarendom Press, Oxford 1996.





\bibitem{Masuda}
K.~Mausuda:
\newblock Weak solutions of Navier-Stokes equations.
\newblock{\em T\^ohoku Math. Journ.} {{\bf 36} (1984),  623-646.}



\bibitem{BrzezniakHausenblasRazafimandimby}
E.~Motyl: 
\newblock Stochastic Navier--Stokes equations driven by L\'evy noise in unbounded $3$d domains.
\newblock{\em Potential Anal}, 38(3), 863--912, 2012.




\bibitem{prodi}
G.~Prodi:
\newblock Un teorema di unicit\`a per le euqazioni di Navier--Stokes.
\newblock {\em Ann. Mat. Pura Appl.}, {48}: 173--182, 1959.




\bibitem{Sch}
V.~Scheffer: 
\newblock An inviscid flow with compact support in space-time.
\newblock {\em J.Geom. Anal. 3(4):343-401, 1993.}



\bibitem{serrin}
J.~Serrin:
\newblock The initial value problem for the Navier-Stokes equations.
\newblock
{\em In Nonlinear Problems 69-98, Univ. of Wisconsin Press. Madison, Wisconsin 1963.}




\bibitem{Shn}
A.~Shnirelman: 
\newblock On the non-uniqueness of weak solution of the Euler equation. 
\newblock{\em Comm. Pure. Appl. Math. 50(12):1261-1286,1997}




\bibitem{sko}
A.~V.~Skorohod:
\newblock Limit theorems for stochastic processes.
\newblock {\em Teor. Veroyatnost. i Primenen}, 1: 289--319, 1956.




\bibitem{emil_01}
L.~Sz{\'e}kelyhidi, Jr., and E.~Wiedemann:
\newblock Young measures generated by ideal incompressible fluid flows.
\newblock {\em Arch. Ration. Mech. Anal.}, {\bf 206}(1): 333--366, 2012.




\bibitem{emil}
E. Wiedemann:
\newblock Weak-strong uniqueness in fluid dynamics.
\newblock {\em https://arxiv.org/pdf/1705.04220.pdf}, 2017.



\end{thebibliography}
\end{document}